\documentclass[10pt]{article}
\usepackage{latexsym, amsmath, amssymb, amsthm}
\usepackage{enumerate, mathrsfs, bbm}
\usepackage{graphicx, tikz}
\usepackage{authblk, setspace, cite}

\usepackage[top=1 in, bottom=1 in, left=1.1 in, right=1.1 in]{geometry}
\usepackage[pdftex,linkcolor=red,citecolor=blue,backref=page]{hyperref}
\numberwithin{equation}{section}

\newtheorem{thm}{Theorem}[section]
\newtheorem{lem}[thm]{Lemma}
\newtheorem{defi}[thm]{Definition}

\newtheorem{pro}[thm]{Proposition}

\newtheorem{cla}[thm]{Claim}

\newcommand{\R}{\mathbb R}

\def\m{\mathbb}    		
\def\eps{\epsilon}      \def\th{\theta}	 \def\d{\delta}  \def\si {\sigma}
  \def\ls{\lesssim}	 \def\ld{\lambda}
\def\la{\langle}  \def\ra{\rangle}  \def\wh{\widehat}
\def\sq{\sqrt}
\def\l{\left}  \def\r{\right}

\def\wh{\widehat}	\def\t{\tilde}
\def\be{\begin{equation}}     \def\ee{\end{equation}}

\def\bp{\begin{pmatrix}}	\def\ep{\end{pmatrix}}

\def\F{\mathcal{F}}

\def\H{\mathcal{H}}

\def\X{\mathcal{X}}
\date{26 March 2018}

\date{}

\title {\bfseries }

\begin{document}
\title{A lower index bilinear estimate for the quadratic Schr\"odinger equation and  application for its half line problem}
\author{	Shenghao Li\\
	School of Mathematical Sciences,\\
	University of Electronic Science and Technology of China,\\
	Chengdu, Sichuan 610054, China \\
	\texttt{lish@uestc.edu.cn} 
	\and
Xin Yang\footnote{Corresponding author}\\
School of Mathematics, Southeast University,\\
 Nanjing, Jiangsu 211189, China\\
\texttt{ xinyang@seu.edu.cn}}
\date{}
\maketitle

\begin{abstract}
	We prove the local well-posedness of the initial boundary value problem for the nonlinear quadratic Schr\"odinger equation under low initial-boundary regularity assumption via the boundary integral operator method introduced by Bona-Sun-Zhang \cite{17}. 
\end{abstract}



\section{Introduction}
In this article, we study the quadratic Schr\"odinger equation posed on the half line,
\begin{equation}\label{nonlinear}
\begin{cases}
i u_t + u_{xx} +u^2= 0, x,t>0\\
u(x,0) = \varphi(x), u(0,t) = h(t),
\end{cases}
\end{equation}
for $(\varphi,h)\in  H^s(\R^+)\times H^{\frac{2s+1}{4}}(\R^+)$ with $s<0$. The nonlinear  Schr\"odinger equation has been derived as models for considerable range of applications in different fields, including shallow water waves, propagation of lights in fiber, Bose-Einstein condensate theory and so on. The study of well-posedness theories of the nonlinear  Schr\"odinger equation, especially for its Cauchy problem, has been continuously drawn attentions and well-developed  during the past decades (See.  \cite{23,51,tao3} and reference therein). On the other hand, research on  initial boundary value problem (IBVP) of dispersive equations has also been largely advanced. To study such problems, Colliander and Kenig \cite{31} introduced a technique that extended the IBVP into a initial value problem; Bona, Sun and Zhang \cite{17} adapted the Laplace transform to establish a boundary integral operator; Fokas  \cite{F1} developed the Fokas unified transform. All of those approaches have now been widely adapted to study   dispersive equations, including, KdV equation, nonlinear Schr\"odinger equation and Boussiensq equation (See. \cite{15-1,ET16,holmer05,M16,F1,F2}). Based on Colliander-Kenig method,  Holmer \cite{holmer05}      consider the IBVP of a generalized nonlinear  Schr\"odinger equation
    \begin{equation}\label{gs}
    \begin{cases}
    i u_t + u_{xx} +\lambda |u|^{\alpha-1} u= 0, x,t>0\\
    u(x,0) = \varphi(x), u(0,t) = h(t),
    \end{cases}
    \end{equation}
and show that the problem is locally well-posed for $0<s<\frac12$ and $2<\alpha<\frac{5-2s}{1-2s}$ with $(\varphi,h)\in  H^s(\R^+)\times H^{\frac{2s+1}{4}}(\R^+)$; Later, Cavalcante \cite{M16}     establish the well-posedness of IBVP \eqref{nonlinear}  for $-\frac34<s<0$. On the other hand, Bona, Sun and Zhang have also studied \eqref{gs} under the $L_t^p$-$L_x^q$ type space and showed $0<s<\frac12$ and $2\leq \alpha<\frac{5-2s}{1-2s}$ through their boundary integral operator method; Erdo\u{g}an and Tzirakis \cite{ET16} considered related problem of the cubic  Schr\"odinger equation
    \begin{equation}\label{cs}
\begin{cases}
i u_t + u_{xx} +  |u|^{2} u= 0, x,t>0\\
u(x,0) = \varphi(x), u(0,t) = h(t),
\end{cases}
\end{equation}
under the Bourgain-type space, $X^{s,b}$, via Bona-Sun-Zhang's idea and prove the result for $0<s<\frac52$. Similar problem has also been studied by Guo and Wu \cite{Guo-1}. 

In the paper, we continue the study of \eqref{nonlinear} and intend to show the local well-posedness for $-\frac34<s<0$ via the boundary integral operator method under Bourgain-type space. To establish our conclusion, there are two main difficulties.
\begin{itemize}
	\item{1.} Extension of the boundary operator from $\R^+$ to $\R$.
	\item{2.} Bilinear estimates under the Bourgain space.
\end{itemize}

The first problem is due to the structure difference between the boundary integral operator and Bourgain spaces. As it has been shown in \cite{BSZ-1}, the key ingredient of the boundary integral operator method is to establish an explicit formula for the problem 
    \begin{equation}\label{bdr}
\begin{cases}
i u_t + u_{xx}  = 0,  \quad x,t>0\\
u(x,0) = 0, u(0,t) = h(t),
\end{cases}
\end{equation}
which can be written as 
\[u(x,t):=[W_{bdr}(h)](x,t) = \frac{1}{\pi} \int_0^\infty e^{-i\rho^2 t}e^{i\rho x} \rho \tilde{h}(-i\rho^2) d\rho + \frac{1}{\pi} \int_0^\infty e^{i\rho^2 t}e^{-\rho x} \rho \tilde{h}(i\rho^2) d\rho,
\]
where $\tilde{h}(s)=\int_0^\infty e^{-st} h(t) \,dt.$ We need to point out that the boundary integral operator $W_{bdr}$ is only defined on $\R^+\times \R^+$. Following is the definition of  related Bourgain spaces for  Shr\"odinger equation which will be used  throughout our article. 
\begin{defi}
	For $s,b\in\R$, $X^{s,b}$  and $Z^{s,b}$ denote the completion of the Schwartz class ${\mathcal S}(\R^2)$ with
\[\|w\|_{X^{s,b}(\R^2)}=\left\|\langle\xi\rangle^s\langle\tau+\xi^2\rangle^b\widehat{w}(\xi,\tau)\right\|_{L^2_{\xi,\tau}(\R^2)},\]
	\[	\|w\|_{Z^{s,b}(\R^2)}=\left\|\langle\tau\rangle^{s/2}\langle\tau+\xi^2\rangle^b\widehat{w}(\xi,\tau)\right\|_{L^2_{\xi,\tau}(\R^2)},\]
	with $\widehat{{ w }}$ denotes the Fourier transform on both  time and space of $w$.
\end{defi}
\noindent One can observe that both $X^{s,b}$ and $Z^{s,b}$ are defined for  $\R\times \R$. Hence, in order to proceed   estimates  on such spaces, one will need a proper extension of $W_{bdr}$ from $\R^+\times \R^+$ to $\R\times \R$. One way is to take a simple ``zero" extension which was first introduced in \cite{ET16} for the Shr\"odinger equation. Similar idea has then applied on other dispersive equations (See. \cite{97,96-1,LCYZ23,Guo-1}). However, according  our early work on a higher order Boussinesq equation \cite{96-1}, such an extension is only valid for Bourgain space with $s>-\frac12$. Thus, to accomplish our goal for $s>-\frac34$, we will need a new extension. Another way to extend the operator $W_{bdr}$ was developed for the KdV equation in \cite{15-1} which is valid for $s<-\frac12$. But, due to the complexity of the method, such an approach has only been applied for KdV-related equations (See. \cite{14,96-1,LCYZ23}). Therefore, we need to   generalize this extension for  the  Shr\"odinger equation and establish corresponding linear estimates.

The second challenge is due to the difference on   Bourgain spaces, $X^{s,b}$, between   initial value problem (IVP) and IBVP. The study of IVP for dispersive equations has been well-developed during the past decades. Thanks to tremendous achievements on studying of KdV and Shr\"odinger equations, Tao \cite{tao1} introduce the multiplier method for general dispersive equations which can provide a delicate estimate for the multilinear estiamtes problem on Bourgain-type space. However, as pointed out in work of the  KdV equation \cite{15-1,31}, the requirement of Bourgain space, $X^{s,b}$, for   IBVP is different from the one for its IVP. While the related estimates for IVP is valid for $b>\frac12$, the IBVP will force $b\leq \frac12$. Such difference  requires  modifications on   multilinear estiamtes problem for the IBVP based on related IVP. Though  bilinear estimates  of the  IBVP for     Schr\"odinger equation has been established in \cite{M16}, we will provide an alternative delicate proof which is triggered  by our earlier work on other dispersive equations (See. \cite{96-1,LCYZ23,XZ-1}). Our approach is originally inspired by Tao's multiplier method, and will be valid for different requirements of $b$. Interested readers can refer to \cite{96-1,LCYZ23} for $b=\frac12$ and \cite{XZ-1} for $b>\frac12$.

Before presenting our result, we first introduce the following notations.
Let us define
\[\X^{s,b}:= C(\R; H^s(\R))\cap X^{s,b}, \]
with the norm
\[\|w\|_{\X^{s,b}}=\l(\sup_{t\in \R} \|w(\cdot,t)\|^2_{H^s(\R)}+\|w\|_{X^{s,b}}\r)^{\frac12}.\]
We also denote  $\X^{s,b}_{T}:=\X^{s,b}(\R^+\times (0,T) )$, $X^{s,b}_{+}:=X^{s,b}(\R^+\times \R^+ )$ and  $\X^{s,b}_{+}:=\X^{s,b}(\R^+\times \R^+ )$.

\begin{thm}\label{main}
	Given $-3/4<s<0$, $\max\{\frac38,\frac18-\frac{s}{2}\}<b<\frac12$ and $r>0$, there exists an $T=T(s,b,r)>0$ such that if $(\varphi,h)\in H^s(\R^+)\times H^{\frac{2s+1}{4}}(\R^+)$ with
	\[\|\varphi\|_{H^s(\R^+)} +\|h\|_{ H^{\frac{2s+1}{4}}(\R^+)}\leq r,\]
	the IBVP \eqref{nonlinear} admits a unique solution $u\in \X_T^{s,b} $
	and the solution map is real analytic. 
\end{thm}
In order to simplify our proof, we will actually establish an alternative version of Theorem \ref{main}. A re-scaling problem will be considered with $$u^{\lambda}(x,t)=\lambda^{2}u(\lambda x,\lambda^{2}t),$$ then the IBVP \eqref{nonlinear} is converted into
\begin{equation}\label{tr-nonlinear}
\begin{cases}
i u^{\lambda}_t + u^{\lambda}_{xx} + (u^\lambda)^2= 0, x,t>0\\
u^{\lambda}(x,0) = \varphi^{\lambda}(x), u^{\lambda}(0,t) = h^{\lambda}(t),
\end{cases}
\end{equation}
where $\varphi^{\lambda}(x)=\lambda^{2}\varphi(\lambda x)$ and $h^{\lambda}(t)=\lambda^{2}h(\lambda^2 t)$. Then, for $\lambda<1$ and $s\geq -1$, one has
\[\|\varphi^\lambda(x)\|_{H^s(\R^+)}\leq \lambda^{\frac32} \|\varphi(x)\|_{H^s(\R^+)}, \quad \|h^\lambda(t)\|_{H^s(\R^+)}\leq \lambda^{\frac32} \|h(t)\|_{H^{\frac{2s+1}{4}}(\R^+)}.\]
Thus, as $\lambda\to 0$, we have 
\[\|\varphi^\lambda(x)\|_{H^s(\R^+)} \to 0, \quad \|h^\lambda(t)\|_{H^s(\R^+)}\to 0.\]
For simplicity in notations, we   neglect   $\lambda$ in \eqref{tr-nonlinear}. The following theorem is equivalent to Theorem \ref{main}.
\begin{thm}\label{alter-main}
	Given $-3/4<s<0$, $\max\{\frac38,\frac18-\frac{s}{2}\}<b<\frac12$ and $T>0$, there exists an $\eps_0:=\eps_0(s,b,T)>0$ such that if $(\varphi,h)\in H^s(\R^+)\times H^{\frac{2s+1}{4}}(\R^+)$ with
	\[\|\varphi\|_{H^s(\R^+)} +\|h\|_{ H^{\frac{2s+1}{4}}(\R^+)}\leq \eps_0,\]
	the IBVP \eqref{tr-nonlinear} admits a unique solution $ u\in \X_T^{s,b} $
	and the solution map is real analytic. 
\end{thm}

The paper is organized as follows. In Section 2, we will re-scale the equation and provide an equivalent  version of our main theorem. Section 3 is devoted to present  thorough explanation  for the extension of operator $W_{bdr}$ from $\R^+\times \R^+$  to $\R\times \R$. We will then establish corresponding estimates based on the extension in Section 4. In Section 5, some existing inequalities and theories will be introduced for later use. Bilinear estimates in related Bourgain spaces will be provded in Section 6. Finally, local well-posedness is established in Section 7.

\section{Extension for the boundary operator}

Consider the following problem
\begin{equation*}
	\begin{cases}
	i u_t + u_{xx} = 0,\\
	u(x,0) = 0, u(0,t) = h(t).
	\end{cases}
\end{equation*}
According to (2.29) in \cite{BSZ-1}, one has
\be\label{bdr op-form}
	[W_{bdr}(h)](x,t) = \frac{1}{\pi} \int_0^\infty e^{-i\rho^2 t}e^{i\rho x} \rho \tilde{h}(-i\rho^2) d\rho + \frac{1}{\pi} \int_0^\infty e^{i\rho^2 t}e^{-\rho x} \rho \tilde{h}(i\rho^2) d\rho,
\ee
where
\be\label{h-t}
\tilde{h}(s)=\int_0^\infty e^{-st} h(t) \,dt.
\ee
Since the operator $W_{bdr}$ is defined only on $\R^+\times \R^+$, we need to properly extend it to $\R^2$ in order to take advantage of the known results in $X^{s,b}$ spaces which are defined on $ \R^2 $. 

Firstly, we apply a change of variable $ \rho\to \mu := \rho^2 $ to (\ref{bdr op-form}) to find
\be\label{bdrop-s1} 
[W_{bdr}(h)](x,t) = \frac{1}{2\pi} \int_0^\infty e^{-i\mu t} e^{i\sqrt{\mu}x}  \tilde{h}(-i\mu) \,d\mu + \frac{1}{2\pi} \int_0^\infty e^{i\mu t} e^{-\sqrt{\mu} x} \tilde{h}(i\mu) \,d\mu. 
\ee
For the first integral, a natural extension can be applied by simply viewing it as a function defined on $ \R^2 $ since the norms of both terms $ e^{-i\mu t} $ and $ e^{i\sqrt{\mu}x} $ are bounded. However, for the second integral, a nontrivial extension has to be introduced since the term $ e^{-\sqrt{\mu} x} $ increases exponentially when $ x<0 $ and $ \mu\to\infty $.
  We thus choose two non-negative functions $ \phi_1 $ and $ \phi_2 $ such that $\text{supp }\phi_1 \subset (-1, 2)$, $\text{supp }\phi_2 \subset (1,\infty)$, and 
$\phi_1(\mu)+\phi_2(\mu)\equiv 1, \quad \mbox{for } \mu\in\R^+$,
to re-write 
\be\label{bdrop-s2}
\int_0^\infty e^{i\mu t} e^{-\sqrt{\mu} x} \tilde{h}(i\mu) \,d\mu = \int_0^2 e^{i\mu t} e^{-\sqrt{\mu} x} \phi_1(\mu) \tilde{h}(i\mu) \,d\mu + \int_1^\infty e^{i\mu t} e^{-\sqrt{\mu} x} \phi_2(\mu) \tilde{h}(i\mu) \,d\mu. \ee
Now the range for $ \mu $ in the first integral on the right hand side of (\ref{bdrop-s2}) is finite, so we can apply a trivial extension to $ \R^2 $ by keeping $ x $ away from $ -\infty $. Therefore, we consider a smooth function $ \phi_3 $ on $ \R $ such that 
\[\phi_3(x)=\begin{cases}
	x, \quad & \mbox{for }\quad x\geq 0, \\
	0,& \mbox{for }\quad x<-1.
\end{cases} \]
Then $ \int_0^2 e^{i\mu t} e^{-\sqrt{\mu}\, \phi_3(x)} \phi_1(\mu) \tilde{h}(i\mu) \,d\mu $ becomes an extension of $\int_0^2 e^{i\mu t} e^{-\sqrt{\mu} x} \phi_1(\mu) \tilde{h}(i\mu) \,d\mu$ to $ \R^2 $. For the second integral on the right hand side of (\ref{bdrop-s2}), the range of $ \mu $ is still infinite, but the value of $ \mu $ is away from 0 thanks to $ \phi_2 $. As a summary so far, we can rewrite $ W_{bdr}(h) $ as follows:
\begin{align}
[W_{bdr}(h)](x,t) =& \frac{1}{2\pi} \int^\infty_0 e^{-i\mu t}e^{i\sqrt{\mu}x}  \tilde{h}(-i\mu) \,d\mu + \frac{1}{2\pi} \int^2_0 e^{i\mu t}e^{-\sqrt{\mu} \, \phi_3(x)} \phi_1(\mu) \tilde{h}(i\mu) \,d\mu \nonumber\\
&	+ \frac{1}{2\pi} \int^\infty_1 e^{i\mu t}e^{-\sqrt{\mu} x} \phi_2(\mu)\tilde{h}(i\mu) d\mu \nonumber\\
:=&I_1(x,t) + I_2(x,t)+I_3(x,t) \label{I},
\end{align}
For $I_1$ or $I_2$, as discussed above, we can apply a natural extension for all $x$ and $t$ by simply viewing them as functions defined on $\R^2$. We denote such extensions of $ I_1 $ and $ I_2 $ as $I_1^*(x,t)$ and $I_2^*(x,t)$ respectively. For $I_3$, more sophisticated construction for extension is desired. The key idea is based on the method in  \cite{15-1} and we will explain the detailed procedure in the following.

Roughly speaking, the extension developed in  \cite{15-1} is a mix of an even extension and an odd extension. The general form of an extension $ I_3^* $ of $ I_3 $ to $ \R^2 $ can be written as below:
\[I^*_3(x,t)=\begin{cases}
I_3(x,t), \quad &x\geq 0,\\
g(x,t), & x<0,
\end{cases}\]
where $g(x,t)$ will be determined later. We denote $ I_{3e} $ and $ g_e $ to be the even extensions of $ I_3 $ and $ g $ respectively, that is, 
\[ I_{3e}(x,t) = \begin{cases}
	I_3(x,t) & \text{if}\quad   x>0,\, t\in\m{R},\\
I_3(-x,t) & \text{if}\quad  x<0,\, t\in\m{R},
\end{cases}  
\qquad g_{e}(x,t) = \begin{cases}
g(-x,t) & \text{if} \quad x>0,\, t\in\m{R},\\
g(x,t) & \text{if}\quad  x<0,\, t\in\m{R}.
\end{cases}
\]

We adopt $ \text{``sgn''} $ to be the standard sign function which is defined as $sgn(x)=1$ if $x\geq 0$ and $sgn(x)=-1$ if $x< 0$.
Then 
\be\label{I3*-decom}\begin{split}
	I^*_3(x,t) 
	&=  \frac12 \big[ I_{3e}(x,t) + g_e(x,t) \big] + \frac12 \text{sgn}(x) \big[ I_{3e}(x,t) - g_e(x,t) \big]. 
\end{split}\ee
We   write $ \wh{I_3^*} $, $ \wh{I_{3e}} $ and $ \wh{g_e} $ to be the space-time Fourier transforms of $ I_3^* $, $ I_{3e} $ and $ g_e $ respectively. It thus can be checked  that 
\be\label{ft-et}
\wh{I_{3e}}(\xi,\tau) = 2\F_t\l[ \int_{\m{R}^+} I_3(x,t) \cos(x\xi) \,dx\r](\tau), \quad 
\wh{g_e}(\xi,\tau) = 2\F_t\l[ \int_{\m{R}^+} g(-x,t) \cos(x\xi) \,dx\r](\tau).
\ee
In addition, recalling $ \F_x(\text{sgn})(\xi) = -\frac{i}{\pi\xi} $, it then follows from (\ref{I3*-decom}) that 
\be\label{I3*-ft1}
\wh{I_3^*}(\xi,\tau) = \frac12\big[ \wh{I_{3e}}(\xi,\tau) + \wh{g_e}(\xi,\tau) \big] - \frac{i}{2\pi}\int_{\m{R}} \frac{1}{\xi-\eta}\big[ \wh{I_{3e}}(\eta,\tau) - \wh{g_e}(\eta,\tau) \big] \,d\eta. 
\ee

Let $ \delta\in (0 ,\frac12]$ be a positive constant which will be specified later, and define
\be\label{Th} \Theta(\xi,\tau) = \Psi(|\xi| -\delta |\tau|^{\frac12} ), \ee
with  $ \Psi\in C^{\infty}(\R) $, $\Psi(x)=1$ if $x\leq 0$ and $\Psi(x)=0$ if $x\geq 0$.
We also set  $ w=w(\tau)$ to be a bounded $ C^{\infty} $ function that will be determined later. Then we choose $ g $ such that 
\be\label{choice of g}
\wh{g_e}(\xi,\tau) + \wh{I_{3e}}(\xi,\tau) = [1-w(\tau)][1-\Theta(\xi,\tau)] \wh{I_{3e}}(\xi,\tau).
\ee
Based on (\ref{choice of g}), we have
\be\label{Th1}
\wh{I_{3e}}(\xi,\tau) - \wh{g_e}(\xi,\tau) = \Theta_{1}(\xi,\tau) \wh{I_{3e}}(\xi,\tau),\quad \mbox{where} \quad \Theta_{1}(\xi,\tau) =  [1+w(\tau)][1-\Theta(\xi,\tau)] + 2\Theta(\xi,\tau)
\ee
Plugging (\ref{choice of g}) and (\ref{Th1}) into (\ref{I3*-ft1}) yields 
\be\label{I3*-ft2} 
\wh{I_3^*}(\xi,\tau) =  \frac12 [1-w(\tau)][1-\Theta(\xi,\tau)] \wh{I_{3e}}(\xi,\tau)  - \frac{i}{2\pi}  \int_{\m{R}} \frac{1}{\xi-\eta} \Theta_{1}(\eta,\tau) \wh{I_{3e}}(\eta,\tau) \,d\eta  : = I_{31}(\xi,\tau) -   I_{32}(\xi,\tau).
 \ee
For simplicity in notations, we will neglect the constants scales in both $I_{31}$ and $I_{32}$.
 Noticing that both $ \Theta_{1}(\eta,\tau) $ and $ \wh{I_{3e}}(\eta,\tau) $ are even functions in $ \eta $, one can write $I_{32}$ as 
\begin{align} 
	I_{32}(\xi,\tau) = &\int_{\R^+} \l(\frac{1}{\xi-\eta} + \frac{1}{\xi+\eta}\r) \Theta_{1}(\eta,\tau) \wh{I_{3e}}(\eta,\tau) \,d\eta  
	= \int_{\R^+} \frac{2\xi}{\xi^2-\eta^2} \Theta_{1}(\eta,\tau) \wh{I_{3e}}(\eta,\tau) \,d\eta. \label{i32}
\end{align}
In the future estimate via the Fourier restriction norms, formula (\ref{i32}) usually works effectively in the region where $ |\xi| $ is small. But in the region where $ |\xi| $ is large, the term $ \frac{2\xi}{\xi^2-\eta^2} $ behaves like $ \frac{1}{\xi} $ which decays too slowly as $ \xi\to\infty $. In the following, we will make a delicate choice of the function $ w $, see (\ref{Th1}), such that the term $ \frac{2\xi}{\xi^2-\eta^2} $ is reduced to $ \frac{2\eta^2}{\xi(\xi^2-\eta^2)}  $ which behaves like $ \frac{1}{\xi^3} $ and decays much faster than $ \frac{1}{\xi} $ as $ \xi\to \infty $. The key observation for this reduction is the following decomposition
\[ \frac{2\xi}{\xi^2-\eta^2} = \frac{2}{\xi} + \frac{2\eta^2}{\xi(\xi^2-\eta^2)},\]
which implies that 
\be\label{I32-1}
	I_{32}(\xi,\tau)= I_{32-1}(\xi,\tau) + I_{32-2}(\xi,\tau),  
\ee
where 
\be\label{I32_12}	
I_{32-1}(\xi,\tau) =  \frac{2}{\xi} \int_{\R^+} \Theta_{1}(\eta,\tau) \wh{I_{3e}}(\eta,\tau) \,d\eta, \quad I_{32-2}(\xi,\tau) = \frac{2}{\xi} \int_{\R^+} \frac{\eta^2}{\xi^2-\eta^2} \Theta_{1}(\eta,\tau) \wh{I_{3e}}(\eta,\tau) \,d\eta.  
\ee
Then our goal is to choose a suitable function $ w $ in (\ref{Th1}) such that $ I_{32-1} $ vanishes. Once this goal is achieved, $ I_{32} $ is reduced to $ I_{32-2} $ which works more effectively in the future estimate when $ |\xi| $ is large. 

Before carrying out the definition of $ w $, we first compute $ \wh{I_{3e}}(\eta,\tau) $.  From (\ref{I}), we have
\be\label{ft}\begin{split}
	\int_{\R^+} I_3(x,t) \cos(x\xi)\,dx &= \frac{1}{2\pi} \int_{1}^{\infty} e^{i\mu t} \phi_2(\mu)\t{h}(i\mu)\bigg( \int_{\R^+} e^{-\sqrt{\mu}x} \cos(x\xi)\,dx \bigg) \,d\mu \\
	&= \frac{1}{2\pi} \int_{1}^{\infty} e^{i\mu t}\phi_2(\mu) \t{h}(i\mu) \frac{\sqrt{\mu}}{\xi^2+\mu}  \,d\mu =  \frac{1}{2\pi} \int_{\R} e^{i\mu t}\phi_2(\mu) \t{h}(i\mu) \frac{\sqrt{\mu}}{\xi^2+\mu}  \,d\mu,
\end{split}\ee
the last equality holds since   the support of $ \phi_2 $ lies in $ (1,\infty) $.
Then we take advantage of (\ref{ft-et}) to find 
\be\label{I_3e-ft}
\wh{I_{3e}}(\xi,\tau) = 2\F_t\l[ \int_{\m{R}^+} I_3(x,t) \cos(x\xi) \,dx\r](\tau) = \frac{1}{\pi}\, \phi_2(\tau) \t{h}(i\tau) \, \frac{\sqrt{\tau}}{\xi^2+\tau}. 
\ee
Consequently, 
\be \label{I32-12alt}
\begin{cases}
I_{32-1}(\xi,\tau) = & \frac{2}{\pi\xi}\, \phi_2(\tau) \t{h}(i\tau) \int_{\R^+} \Theta_{1}(\eta,\tau) \frac{\sqrt{\tau}}{\eta^2+\tau} \,d\eta,  \\
I_{32-2}(\xi,\tau) = & \frac{2}{\pi\xi}\, \phi_2(\tau) \t{h}(i\tau) \int_{\R^+} \frac{\eta^2}{\xi^2 - \eta^2} \Theta_{1}(\eta,\tau) \frac{\sqrt{\tau}}{\eta^2+\tau} \,d\eta.  
\end{cases}
\ee
Since $ \text{supp }\phi_2\subset (1,\infty) $, we only need to consider the case when $ \tau\geq 1 $. Denote 
\be\label{f12}
f_1(\tau) = \int_{\R^+} \frac{\sqrt{\tau}}{\eta^2+\tau} [1-\Theta(\eta,\tau)] \,d\eta, \quad f_2(\tau) = \int_{\R^+}  \frac{\sqrt{\tau}}{\eta^2+\tau} \Theta(\eta,\tau) \,d\eta, \quad\forall \, \tau\geq 1.
\ee
Then 
$I_{32-1}(\xi,\tau) =  \frac{2}{\pi\xi}\, \phi_2(\tau) \t{h}(i\tau) \big([1+w(\tau)] f_1(\tau)+2f_2(\tau)\big)$. 
We intend to define $ w=w(\tau) $ such that 
\be\label{choice of w}
 [1+w(\tau)] f_1(\tau) = -2 f_2(\tau),  \quad \forall\, \tau\geq 1, \quad \mbox{ and } \quad  
 w(\tau) = 0,  \quad\forall\, \tau\leq \frac12.
\ee
Through a direct computation, we can check that $ f_1(\tau) $ has a positive lower bound (uniformly in $ \tau\geq 1 $). Thus, $ w $ is  well-defined and $ w\in C^{\infty}(\R) \cap L^{\infty}(\R) $.   It then follows that $I_{32-1}(\xi,\tau) = 0$ for such $w(\tau)$. 


Finally, by summarizing up all the above results, we obtain the desired extension for $[W_{bdr}(h)](x,t)$. We denote this extension as $ [\Phi_{bdr}(h)](x,t) $ which is well-defined on $ \R^2 $ and matches $ [W_{bdr}(h)](x,t) $ on $ \R^+\times \R^+ $. The formula of $ [\Phi_{bdr}(h)](x,t)  $ consists of three parts: 
\be\label{bdrop-f}
 [\Phi_{bdr}(h)](x,t) = I_1^*(x,t) + I_2^*(x,t) + I_3^*(x,t),
\ee
where 
\[I_1^*(x,t) = \frac{1}{2\pi} \int^\infty_0 e^{-i\mu t}e^{i\sqrt{\mu}x}  \tilde{h}(-i\mu) \,d\mu, \qquad I_2^*(x,t) = \frac{1}{2\pi} \int^2_0 e^{i\mu t}e^{-\sqrt{\mu} \, \phi_3(x)} \phi_1(\mu) \tilde{h}(i\mu) \,d\mu, \]
and $ I_3^* $ is defined via its space-time Fourier transform: $$ \wh{I_3^*}(\xi,\tau) = I_{31}(\xi,\tau) -   I_{32}(\xi,\tau), $$ where $ I_{31} $ can be written as (see (\ref{I3*-ft2}) and (\ref{I_3e-ft}))
\begin{align}
& I_{31}(\xi,\tau) = \frac{1}{2\pi} \phi_2(\tau)\, \t{h}(i\tau)\, [1-w(\tau)] \,  [1-\Theta(\xi,\tau)] \, \frac{\sqrt{\tau}}{\xi^2+\tau},   \label{I31gs} 
\end{align}
and $I_{32}(\xi,\tau) $ can be written as (see (\ref{I32_12}) and (\ref{I32-12alt})) either 
\be\label{I32-2gs}
I_{32}(\xi,\tau)= \frac{2}{\pi}\, \phi_2(\tau)\, \t{h}(i\tau) \int_{\R^+} \frac{\xi}{\xi^2-\eta^2} \Theta_{1}(\eta,\tau) \frac{\sqrt{\tau}}{\eta^2+\tau} \,d\eta,
\ee
or 
\be\label{I32-2gs-alternative}
I_{32}(\xi,\tau)= \frac{2}{\pi\xi}\, \phi_2(\tau)\, \t{h}(i\tau) \int_{\R^+} \frac{\eta^2}{\xi^2-\eta^2} \Theta_{1}(\eta,\tau) \frac{\sqrt{\tau}}{\eta^2+\tau} \,d\eta.
\ee

\section{Estimates for the boundary operator}

The goal of this section is to establish the following estimates for the boundary operator $ \Phi_{bdr} $. Here and after, we denote $\th(t)$ to be a $ C^\infty-$function with $\th(t)=1$ for $t\in (-1,1)$ and supp $\th\in (-2,2)$. 
\begin{pro}\label{bdrx}
	For $-\frac32<s< \frac32$ and $b<\frac12$, one has
	\begin{align}
				\big\| \th (t)[\Phi_{bdr}(h)](x,t) \big\|_{X^{s,b}} & \lesssim \| h \|_{H^{\frac{2s+1}{4}}(\R^+)}.  \label{phi1}
	\end{align}
\end{pro}

\begin{pro}\label{bdrh}
	For $-\frac32<s< \frac12$, one has
	\begin{align}
	\sup_{t\in \R} \big\| [\Phi_{bdr}(h)](x,t) \big\|_{H^s_x(\R)} & \lesssim \| h \|_{H^{\frac{2s+1}{4}}(\R^+)}. 			\label{phi2}
	\end{align}
\end{pro}

Before proving the above propositions, we will introduce some lemmas that are needed  in the later study. Firstly, in order to work with the Sobolev norms, it is more convenient to consider the Fourier transform rather than the Laplace transform. Note that,  for  function $ f\in H^s(\m{R}) $, its Fourier transform, $ \hat{f} $, is defined as 
\[ \hat{f}(\xi) = \int_{\m{R}} e^{-i \xi t} f(t) \,dt, \quad\forall\, \xi\in\m{R}. \]
For any function $ h\in H^s(\m{R}^+) $, we denote $ h_* $ to be its zero extension, that is, 
\be\label{ze}
h_*(x) = \begin{cases}
h(x) & \text{if} \quad x>0, \\
0 & \text{if} \quad x<0.
\end{cases} \ee
Then the Laplace transform of $ h $ can be connected with the Fourier transform of $ h^* $ in the following way:
\[ \t{h}(i\tau) = \int_0^\infty e^{-i\tau t} h(t) \,dt = \int_{\m{R}} e^{-i\tau t} h_*(t) \,dt = \wh{h_*}(\tau), \quad\forall\, \tau\in\m{R}. \]

\begin{lem}\label{Lemma, est-ze}
Let $ h\in H^s(\m{R}^+) $ and denote $ h_* $ to be the zero extension of $ h $ as defined in (\ref{ze}). Assume one of the following conditions hold:
 \[-1/2 < s < 1/2  \qquad \mbox{or} \qquad 
 1/2 < s < 3/2  \quad \mbox{ and } \quad  h(0) = 0 \]	
Then $ h_*\in H^s(\m{R}) $ and there exists a constant $ C=C(s) $ such that 
\[ \| h_* \|_{H^s(\m{R})} \leq C \| h \|_{H^s(\m{R}^+)}. \]
\end{lem}
\begin{proof}
This result is standard, see e.g. \cite{JK95} or Lemma 2.1(i)(ii) in \cite{ET16}.  
\end{proof}

\begin{lem}\label{xsb}
	Let $s\in \R$ and $0<b\leq 1$, then
	\[\|\th(t) h\|_{X^{s,b}} \lesssim \|h\|_{X^{s,b}}\]
\end{lem}
\noindent Such lemma would be useful later, and the proof can be referred  from Lemma 4.2 of \cite{15-1}. The following lemmas are essential to  establish Proposition \ref{bdrx}.
\begin{lem}\label{Lemma, I31}
	Assume $-\frac{s}{2} - \frac14 < b<-\frac{s}{2}+\frac34$, then there exists a constant $ C=C(s,b) $ such that 
	\be\label{eI31}
	\int_\R \int_\R \la \xi\ra^{2s} \la \tau+\xi^2\ra^{2b} |I_{31}(\xi,\tau)|^2 d\xi d\tau \leq C \|h\|_{H^{\frac{2s+4b-1}{4}}(\R^+)},
	\ee
	where $ I_{31} $ is given in \eqref{I31gs}.
\end{lem}

\begin{proof}[\bf Proof of Lemma \ref{Lemma, I31}]
	According to the formula of $I_{31}$ in (\ref{I31gs}), one has
	\begin{align*}
		\mbox{LHS of \eqref{eI31}} = & \frac{1}{4\pi^2} \int_\R \int_\R \la \xi\ra^{2s}\la \tau+\xi^2\ra^{2b} \phi_2^2(\tau)\, |\wh{h_*}(\tau)|^2 \, [1-w(\tau)]^2 \, [1-\Theta(\xi,\tau)]^2\, \frac{\tau}{(\xi^2+\tau)^2} \,d\tau \,d\xi\\
		= & \frac{1}{4\pi^2} \int_\R \phi_2^2(\tau) |\wh{h_*}(\tau)|^2 \, [1-w(\tau)]^2 \, \tau \,\bigg( \int_{\m{R}} \la \xi \ra^{2s} \la \tau+\xi^2\ra^{2b-2} [1-\Theta(\xi,\tau)]^2 \,d\xi\bigg)  \,d\tau.
	\end{align*}
Since the support of $ \phi_2 $ lies in $ (1,\infty) $ and $ \phi_2 $ is bounded, we obtain 
\be\label{fr-I31-e} 
\mbox{LHS of \eqref{eI31}} \ls \int_{1}^{\infty} |\wh{h_*}(\tau)|^2 \, \tau D(\tau) \,d\tau,
\ee
where for $ \tau\geq 1 $,
	\be\label{D-tau}\begin{split}
	D(\tau) & := \int_\R  \la \xi \ra^{2s} \la \tau+\xi^2\ra^{2b-2}  [1-\Theta(\xi,\tau)]^2 \, d\xi    = 2 \int_0^{\infty}  \la \xi \ra^{2s} \la \tau+\xi^2\ra^{2b-2}   [1-\Theta(\xi,\tau)]^2 \, d\xi.
	\end{split}\ee
The last equality holds  due to the fact that $ \Theta(\xi,\tau) $ is even in $ \xi $. Note that $1-\Theta(\xi,\tau)\equiv0$ if $ \xi\leq \delta \tau^{\frac12} $ with $\xi\geq 0$ and $\tau>1$.  Thus, we will only need to consider the case for $\xi>\delta \tau^{\frac12}$ which infers that 
\[D(\tau) = 2\int^\infty_{\delta \tau^\frac12} \la \xi \ra^{2s} \la \tau+\xi^2\ra^{2b-2} [1-\Theta(\xi,\tau)]^2 \,d\xi. \]
Next,  by the change of variable $ \xi\to \mu:= \xi^2 $,  we obtain
	\begin{align*}
		D(\tau) \lesssim & \int^\infty_{\delta^2 \tau} \la \mu \ra^{s} \la \tau+\mu\ra^{2b-2} \mu^{-\frac12} \,d\mu=\l( \int_{\delta^2 \tau}^{2\tau}+\int^\infty_{2\tau}\r)\la \mu \ra^{s} \la \tau+\mu\ra^{2b-2} \mu^{-\frac12} \,d\mu
		:= D_1(\tau) + D_2(\tau).
	\end{align*}
For the term $D_1$, due to $\mu>\delta^2 \tau$, $\tau\geq 1$ and $\tau \sim \mu$, it leads to $D_1\lesssim \tau^{s+2b-\frac32}.$
For the term $D_2$, one has
\[D_2 = \int^\infty_{2\tau}\la \mu \ra^{s} \la \tau+\mu\ra^{2b-2}   \mu^{-\frac12} \,d \mu \lesssim \int^\infty_{2\tau}\la \mu \ra^{s+2b-\frac52} \, d \mu. \]
Since it is assumed that $ b < -\frac{s}{2}+\frac34 $, which implies $ s+2b-\frac52 < -1 $, one has $ D_2 \ls \tau^{s+2b-\frac32}$. Combining  estimates for $ D_1 $ and $ D_2 $, we deduce 
\be\label{D-est}
|D(\tau)| \ls \, \tau^{s+2b-\frac32}, \quad \forall \,\tau\geq 1.
\ee

Now we plug (\ref{D-est}) into (\ref{fr-I31-e}) to find
	\[
	\mbox{LHS of \eqref{eI31}} \lesssim \int_{1}^{\infty} \tau^{s+2b-\frac12} |\wh{h_*}(\tau)|^2 \,d\tau \lesssim \|h_*\|_{H^{\frac{2s+4b-1}{4}}(\R)}\ls  \|h\|_{H^{\frac{2s+4b-1}{4}}(\R^+)}. 
	\]
The last inequality holds due to Lemma \ref{Lemma, est-ze} and the assumption $-\frac{s}{2} - \frac14 < b<-\frac{s}{2}+\frac34$ in Lemma \ref{Lemma, I31} which implies that $ -\frac12 < \frac{2s+4b-1}{4} < \frac12 $. Thus, Lemma \ref{Lemma, I31} is verified.
\end{proof}

\begin{lem}\label{Lemma, I32}
	Assume $s>-\frac32$ and $ -\frac{s}{2}-\frac14< b<-\frac{s}{2}+\frac{3}{4}$. Let $0<\eps_0< \frac{3-2s-4b}{4}$,
	then there exists $ C=C(s,b,\eps_0) $ such that
	\be\label{eI32}
	\int_\R \int_\R \la \xi\ra^{2s} \la \tau+\xi^2\ra^{2b} |I_{32}(\xi,\tau)|^2 d\xi d\tau \leq C \|h\|_{H^{\frac{2s+4b-1}{4}+\eps_0}(\R^+)},
	\ee
	with $ I_{32} $  given in \eqref{I32-2gs} and \eqref{I32-2gs-alternative}.
\end{lem}

\begin{proof}[\bf Proof of Lemma \ref{Lemma, I32}]
Since the support of $ \phi_2(\tau) $ is in $ (1,\infty) $, then $ I_{32}(\xi,\tau) = I_{32-2}(\xi,\tau) = 0 $ whenever $ \tau\leq 1 $, so we only need to consider the region where $ \tau\geq 1 $. In addition, since $ I_{32}(\xi,\tau)$ is odd in $ \xi $, we have 
\be\label{I32-2s}
\text{LHS of (\ref{eI32})} = 2 \int_1^\infty \int_{\R^+} \la \xi\ra^{2s} \la \tau+\xi^2\ra^{2b} |I_{32}(\xi,\tau)|^2 \,d\xi \,d\tau.
\ee
where $  I_{32}(\xi,\tau) $ is given by (\ref{I32-2gs}) and (\ref{I32-2gs-alternative}) which can be written as $ I_{32}(\xi,\tau) = \frac{2}{\pi} \phi_2(\tau)\, \wh{h_*}(\tau) F(\xi,\tau)$, with 
\be\label{F}
F(\xi,\tau) := \left\{\begin{aligned}
	\int_{\R^+} \frac{\xi}{\xi^2 - \eta^2}\frac{\sqrt{\tau}}{\eta^2 + \tau} \Theta_1(\eta,\tau) \,d\eta, \quad \text{for} \quad \tau\geq 1, \, 0\leq \xi< 2\sqrt{\tau}, \\
	\int_{\R^+} \frac{\eta^2}{\xi(\xi^2 - \eta^2)}\frac{\sqrt{\tau}}{\eta^2 + \tau} \Theta_1(\eta,\tau) \,d\eta,  \quad \text{for} \quad \tau\geq 1, \, \xi\geq 2\sqrt{\tau}.
\end{aligned}\right.\ee
We remark that the function $ F $ is well-defined since
$\int_0^\infty \frac{1}{\xi^2-\eta^2} \,d\eta = 0$
as a principal value integral.
Consequently, we deduce from (\ref{I32-2s}) that 
\be\label{I32-est-byE}
\begin{aligned}
	\text{LHS of (\ref{eI32})} &\ls  \int_{1}^{\infty}   |\wh{h_*}(\tau)|^2 \bigg( \int_{\R^+} \la \xi\ra^{2s} \la \tau+\xi^2\ra^{2b}  F^2(\xi,\tau) \,d\xi \bigg) \,d\tau  &\ls \int_{1}^{\infty} |\wh{h_*}(\tau)|^2\, E(\tau) \,d\tau,
\end{aligned}\ee
where 
\be\label{E}
E(\tau) := \int_{\R^+} \la \xi\ra^{2s} \la \tau+\xi^2\ra^{2b} F^2(\xi,\tau) \,d\xi, \quad\forall\, \tau\geq 1.
\ee
It can be checked that
\be\label{F-est0} 
F(\xi,\tau) \ls \left\{ \begin{array}{lll} \vspace{0.1in}
\xi(1 + \ln \tau)/\tau & \text{if} & \tau\geq 1, \,0\leq \xi<2\sqrt{\tau}, \\
 \tau/\xi^3& \text{if} & \tau\geq 1, \, \xi\geq 2\sqrt{\tau}.	
\end{array}\right.	
\ee
The justification of (\ref{F-est0}) is tedious, to avoid interrupting the flow of the current presentation, we will postpone the verification of (\ref{F-est0}) to Lemma \ref{Lemma, F-est}.  

Next, let us move on to the estimate for $E(\tau)$. Plugging (\ref{F-est0}) into (\ref{E}) leads to $$ E(\tau) \leq E_1(\tau) + E_2(\tau), $$ where 
\[E_1(\tau) = \bigg(\frac{1+\ln\tau}{\tau}\bigg)^2 \int_0^{2\sqrt{\tau}} \la \xi\ra^{2s} \la \tau+\xi^2\ra^{2b} \xi^2 \,d\xi, \quad 
	E_2(\tau) = \tau^2 \int_{2\sqrt{\tau}}^{\infty} \la \xi\ra^{2s} \la \tau+\xi^2\ra^{2b} \xi^{-6} \,d\xi.\]
For $ E_1 $, it follows from $ \tau\geq 1 $ that 
\[
E_1(\tau) \ls \bigg(\frac{1+\ln\tau}{\tau}\bigg)^2 \bigg( \int_0^1 \tau^{2b} \,d\xi + \int_1^{2\sqrt{\tau}} \xi^{2s+2} \tau^{2b} \,d\xi \bigg).
\]
Since $ s>-\frac32 $, then we find 
$E_1(\tau) \ls \tau^{s+2b-\frac12} (1+\ln\tau)^2$.
While for $ E_2 $, one can compute that $ E_2(\tau) \ls \tau^{s+2b-\frac12} $ for  $ b< -\frac{s}{2}+\frac54 $.
Hence, for any $\eps_0 >0$,
\be\label{E-est} 
E(\tau) \leq E_1(\tau) + E_2(\tau) \ls \tau^{s+2b-\frac12} (1+\ln\tau)^2\ls \tau^{s+2b-\frac12+2\eps_0}, \quad\forall\, \tau\geq 1.
\ee 

Finally, according to  (\ref{I32-est-byE}) and (\ref{E-est}), one has  
\begin{align*}
	\text{LHS of (\ref{eI32})} &\ls \int_1^{\infty}  |\wh{h_*}(\tau)|^2 \tau^{s+2b-\frac12+2\eps_0} \,d\tau  \ls  \|h_*\|_{H^{\frac{2s+4b-1}{4} +\eps_0 }(\R)} \ls \|h\|_{H^{\frac{2s+4b-1}{4}+\eps_0}(\R^+)},
\end{align*}
where the last inequality is due to Lemma \ref{Lemma, est-ze} and the fact that 
$-\frac12 < \frac{2s+4b-1}{4} + \eps_0 < \frac12$.
\end{proof}

Now we state and prove the result on the estimate of $ F(\xi, \tau) $ that was used in the proof of the above Lemma \ref{Lemma, I32}.

\begin{lem}\label{Lemma, F-est}
The function $ F $, as defined in (\ref{F}), satisfies the following estimates.
\be\label{F-est} 
	F(\xi,\tau) \ls \left\{ \begin{array}{lll} \vspace{0.1in}
\xi(1 + \ln \tau)/\tau & \text{if} & \tau\geq 1, \, 0<\xi<2\sqrt{\tau}, \\
	\tau/\xi^3 & \text{if} & \tau\geq 1, \, \xi\geq 2\sqrt{\tau}.
	\end{array}\right.	
\ee
\end{lem}

\begin{proof}
Recalling the definition of $F(\xi,\tau)$ in (\ref{F}), we first apply  change in variables as $ \eta \to n:=\eta/\sqrt{\tau}$  to find
\be\label{F-cv} 
F(m\sq{\tau},\tau) = \left\{ \begin{aligned} \vspace{0.1in}
	\frac{m}{\sq{\tau}} \int_{0}^{\infty} \frac{1}{m^2 - n^2}\frac{1}{n^2 + 1} \Theta_1(n\sq{\tau},\tau) \,dn, &\quad\forall\, 0 \leq m<2, \, \tau\geq 1, \\
	\frac{1}{m\sq{\tau}} \int_{0}^{\infty} \frac{n^2}{m^2 - n^2}\frac{1}{n^2 + 1} \Theta_1(n\sq{\tau},\tau) \,dn, &\quad\forall\, m\geq2, \, \tau\geq 1.	
\end{aligned}\right.	
\ee
where $ m:=\xi/\sqrt{\tau} $. Thus, the targeted estimate (\ref{F-est}) is converted to 
\be\label{F-est1} 
F(m\sq{\tau},\tau) \ls \left\{ \begin{array}{lll} \vspace{0.1in}
m(1 + \ln \tau)/\sq{\tau} & \text{if} & 0\leq m<2,\, \tau\geq 1, \\
	1/(m^3\sq{\tau}) & \text{if} & m\geq 2,\, \tau\geq 1.	
\end{array}\right.	
\ee
Next, we will estimate $ F(m\sqrt{\tau}, \tau) $ by dividing the argument into two cases depending on the size of $ m $.
	
\begin{itemize}
		
	\item{\bf Case 1: $ m\geq 2 $}. 
		
	Based on the definition of $ \Theta_1 $ in (\ref{Th1}), we decompose $ F $ in (\ref{F-cv}) as below: 
	\be\label{Ffj}
	F(m\sq{\tau}, \tau) = F_1(m, \tau) + F_2(m, \tau), \quad\forall\, m\geq 2, \, \tau\geq 1,
	\ee
	where 
	\begin{align}
		F_1(m, \tau) :=  \frac{2}{m\sq{\tau}} \int_{0}^{\infty} \frac{n^2}{m^2 - n^2}\frac{1}{n^2 + 1} \Theta(n\sq{\tau},\tau) \,dn, \label{F1} 
			\end{align}
			\begin{align} 
		F_2(m, \tau) := \frac{1+w(\tau)}{m\sq{\tau}} \int_{0}^{\infty} \frac{n^2}{m^2 - n^2}\frac{1}{n^2 + 1} \big[ 1 - \Theta(n\sq{\tau},\tau)\big] \,dn. \label{F2}
	\end{align}
	Based on the definition (\ref{Th}) for $ \Theta $, we have
	\[ 
	\Theta(n\sq{\tau},\tau) = \Psi\big((n-\delta)\sq{\tau}\big) 
	=\begin{cases}
		1 & \text{if} \quad n \leq \delta, \\
		0 & \text{if} \quad n \geq \delta + \frac{1}{\sq{\tau}},
	\end{cases}
	\]
	where $ \Psi $ is the smooth cutoff function as defined in (\ref{Th}). Denote
	$a_1 = \delta + 1/\sq{\tau}$.
	Then (\ref{F1}) and (\ref{F2}) are reduced to the following.
	\begin{align}
		F_1(m, \tau) &=  \frac{2}{m\sq{\tau}} \int_{0}^{a_1} \frac{n^2}{m^2 - n^2}\frac{1}{n^2 + 1} \Psi\big((n-\delta)\sq{\tau}\big)  \,dn, \label{F1s} \\
		F_2(m, \tau) &= \frac{1+w(\tau)}{m\sq{\tau}} \int_{\delta}^{\infty} \frac{n^2}{m^2 - n^2}\frac{1}{n^2 + 1} \big[ 1 - \Psi\big((n-\delta)\sq{\tau}\big) \big] \,dn.  \label{F2s}
	\end{align}
	
	We first estimate $ F_1 $. 	Since $ \tau\geq 1 $ and $ \delta\in (0,\frac12] $, we know that 
	$\delta < a_1 \leq \delta + 1 \leq \frac32$.
	Consequently, $ 0\leq n\leq a_1\leq \frac32 $. 	On the other hand, $ m $ is assumed to be at least 2 in this case, so 
	$m^2-n^2  > \frac{1}{4}m^2$.
	Then it follows from (\ref{F1s}) that 
	\[ | F_1(m, \tau) | \ls 1/(m^3\sq{\tau}). \]
	
	Next, we work with $ F_2 $ which can be decomposed as:
	\begin{align*}
		F_2(m, \tau)= &  \frac{1+w(\tau)}{m\sq{\tau}} \bigg( \int_{\delta}^{a_1} + \int_{a_1}^{\infty} \bigg) \frac{n^2}{m^2 - n^2}\frac{1}{n^2 + 1} \big[ 1 - \Psi\big((n-\delta)\sq{\tau}\big) \big] \,dn \\
	:= 	&F_{21}(m, \tau) + F_{22}(m, \tau).
	\end{align*}
	The estimate on $ F_{21} $ is similar to that on $ F_{1} $, and we obtain 
	$|F_{21}(m, \tau) | \ls  1/(m^3\sq{\tau})$. 
	For $ F_{22} $, we note that $\Psi\big((n-\delta)\sq{\tau}\big) \equiv 0$ for $n>a_1$. Thus, it leads to
	\begin{align*}
		| F_{22}(m, \tau) | \lesssim & \frac{1}{m \sq{\tau}} \l| \int^\infty_{a_1} \frac{n^2}{(m^2-n^2)(n^2+1)} \,dn \r|.
	\end{align*} 	
	Since $\int_0^\infty \frac{1}{m^2-n^2} \,dn = 0$ as a principal value integral, we deduce that 
	\begin{align*}
		\int^\infty_{a_1} \frac{n^2}{(m^2-n^2)(n^2+1)} \,dn &= \int_{a_1}^{\infty} \frac{1}{m^2-n^2}\bigg[ \frac{n^2}{n^2+1} - \frac{m^2}{m^2+1} \bigg] \,dn + \frac{m^2}{m^2+1} \int_{a_1}^{\infty} \frac{1}{m^2-n^2} \,dn \\
		&= - \frac{1}{m^2+1} \int_{a_1}^{\infty} \frac{1}{n^2+1} \,dn - \frac{m^2}{m^2+1} \int_{0}^{a_1} \frac{1}{m^2 - n^2} \,dn.
	\end{align*}
	Since $ \int_{0}^\infty \frac{1}{n^2+1} \,dn \leq \frac{\pi}{2}$, then it is readily seen that 
	 $\l| \frac{1}{m^2+1} \int_{a_1}^{\infty} \frac{1}{n^2+1} \,dn \r| \ls \frac{1}{m^2}$. 
	On the other hand, when $ 0<n<a_1 $, we have $ m^2-n^2 > \frac{m^2}{4} $ and therefore,
	$\int_{0}^{a_1} \frac{1}{m^2-n^2} \,dn  \ls \frac{1}{m^2}$. 
	As a consequence, 
	\[ \bigg| \int_{a_1}^\infty \frac{n^2}{(m^2-n^2)(n^2+1)} \,dn \bigg| \ls \frac{1}{m^2},\] which further implies $ |F_{22}(m,\tau)| \ls \frac{1}{m^3\sq{\tau}}$ and $ |F_2(m,\tau)| \ls \frac{1}{m^3\sq{\tau}} $. Therefore, one has
	
	\be\label{F-est-Case1} 
	|F(m\sq{\tau}, \tau)| \ls1/(m^3\sq{\tau}), \quad\forall\, m\geq 2, \, \tau\geq 1.
	\ee
		
	\item {\bf Case 2:  $0<m\leq 2$.} 
	

	Similar to the derivations from (\ref{Ffj}) to (\ref{F2s}) in Case 1, we can split $ F $ into two parts: 
	\[ F(m\sq{\tau}, \tau) = F_3(m,\tau) + F_4(m,\tau), \]
	where 
	\begin{align}
		F_3(m, \tau) &=  \frac{2m}{\sq{\tau}} \int_{0}^{a_1} \frac{1}{(m^2 - n^2)(n^2 + 1)} \Psi\big((n-\delta)\sq{\tau}\big) \,dn, \label{F3} \\
		F_4(m, \tau) &= \frac{m [1+w(\tau)]}{\sq{\tau}} \int_{\delta}^{\infty} \frac{1}{(m^2 - n^2)(n^2 + 1)} \big[ 1 - \Psi\big((n-\delta)\sq{\tau}\big) \big] \,dn.  \label{F4}
	\end{align}
Moreover, since $\frac{1}{(m^2-n^2)(n^2+1)} = \frac{1}{m^2+1} \Big( \frac{1}{m^2-n^2} + \frac{1}{n^2+1} \Big)$,
$F_{3}$ and $F_{4}$ can be further decomposed as 
	\be\label{F3-F4}\left\{\begin{aligned}
		F_{3}(m,\tau) &= \frac{2m}{(m^2+1)\sq{\tau}}\, \big( F_{31}(m,\tau) + F_{32}(m,\tau)\big), \\
		F_{4}(m,\tau) &= \frac{m [1+w(\tau)]}{(m^2+1) \sq{\tau}} \,\big( F_{41}(m,\tau) + F_{42}(m,\tau)\big),	
	\end{aligned}\right.\ee
	where 
	\begin{equation*}
		F_{31}(m,\tau) =  \int_0^{a_1} \frac{1}{m^2-n^2} \, \Psi\big((n-\delta)\sq{\tau}\big)  \,dn,  \quad
		F_{32}(m,\tau) = \int_0^{a_1} \frac{1 }{n^2+1}\, \Psi\big((n-\delta)\sq{\tau}\big)  \,dn,	
	\end{equation*}
		\begin{equation*}
		F_{41}(m,\tau) = \int_{\delta}^{\infty} \frac{1}{m^2-n^2} \,\big[ 1 - \Psi\big((n-\delta)\sq{\tau}\big)  \big] \,dn, \quad
		F_{42}(m,\tau) = \int_{\delta}^{\infty} \frac{1 }{n^2+1}\, \big[ 1 - \Psi\big((n-\delta)\sq{\tau}\big) \big] \,dn.
		\end{equation*}
	
	Since $ \Psi $ is bounded between $ 0 $ and $ 1 $, one can easily verify that 
	\be\label{F32-F42-est}
	| F_{32} (m,\tau) | + | F_{42} (m,\tau) | \lesssim  1.
	\ee
	We then consider the estimate for $F_{31}$, while $F_{41}$ can be addressed similarly. 	For ease of notation, for fixed $ \delta\in(0,\frac12] $ and $ \tau\geq 1 $, we define a $C^\infty$ function $ f $ as 
	\[ f(x) = \Psi\big( (x-\delta)\sqrt{\tau} \big), \quad \forall\, x\in\m{R}, \quad \mbox{with} \quad f(x)= \left\{\begin{array}{lll}
	1 & \text{if}  & x\leq \delta, \\
	0 & \text{if} & x\geq a_1.
	\end{array}\right. \]
	where
		\be\label{f-mv}
	|f(x)-f(y)| \leq B |x-y| \sqrt{\tau}, \quad \forall\, x,y\in\m{R},
	\ee
	where $B := \sup\limits_{x\in \R} |\Psi'(x)|$. Meanwhile, one can rewrite the term $F_{31}$ as 
	\be\label{F31}
	F_{31}(m,\tau) = \int_0^{a_1} \frac{1}{m^2-n^2}\, f(n) \,dn,
	\ee
 which will be estimated based on four subcases depending on the size of $ m $.
		
	\begin{itemize}
		\item {\bf Case 2.1: $ a_1\leq m\leq 2 $.}
		
		In this case, $ f(m)=0 $ since $ m\geq a_1 $. So we can split the integral in (\ref{F31}) as 
		\[ |F_{31}(m,\tau)| \leq \int_0^{a_1-\frac{1}{\sqrt{\tau}}} \frac{1}{m^2 - n^2}\, f(n) \,dn + \int_{a_1 - \frac{1}{\sqrt{\tau}}}^{a_1} \frac{1}{m^2 - n^2}\, [f(n)-f(m)] \,dn.  \] 
		For the first integral, noticing that $ m+n>\delta $ and $ f(n)\in[0,1] $, we deduce  that
		\[ \int_0^{a_1-\frac{1}{\sqrt{\tau}}} \frac{1}{m^2 - n^2}\, f(n) \,dn \leq \frac{1}{\delta}  \int_0^{a_1-\frac{1}{\sqrt{\tau}}} \frac{1}{m-n} \,dn\ls 1+\ln\tau.  \]
		For the second integral, thanks to the estimate $ |f(n)-f(m)|\leq B(m-n)\sqrt{\tau} $, we obtain
		\[ \int_{a_1-\frac{1}{\sqrt{\tau}}}^{a_1} \frac{1}{m^2 - n^2}\, [f(n)-f(m)] \,dn \leq B\sqrt{\tau} \int_{a_1-\frac{1}{\sqrt{\tau}}}^{a_1} \frac{1}{m+n}\,dn \ls 1. \]
		Hence, $ |F_{31}(m,\tau)| \ls 1+\ln\tau $.
		
		\item {\bf Case 2.2: $ a_1 -  \frac{1}{4\sqrt{\tau}} \leq m \leq a_1 $.}
		
		This is the most challenging case since the singularity $ n=m $ is inside the integration interval $ [0,a_1] $ and is very close to the endpoint $ a_1 $. We delicately split the interval into three parts:
		\[ [0,a_1] = \Big[0,2m-a_1-\frac{1}{2\sqrt{\tau}} \Big] \cup \Big[2m-a_1-\frac{1}{2\sqrt{\tau}}, 2m-a_1\Big] \cup [2m-a_1,a_1]:=\cup_{i=1}^{3}A_i.\]
		Then, we consider the integral of \eqref{F31} in related intervals accordingly and write
		 \begin{equation}\label{J1-j3}
		 \begin{aligned}
		 	 	F_{31}(m,\tau)=   J_1 (m,\tau)+ J_2(m,\tau) + J_3(m,\tau).
		 \end{aligned}
		 \end{equation}
		 For $ J_1 $, since $ \frac{1}{m+n}<\frac{1}{\delta} $ and $ f(n)\leq 1 $, then 
		 \[ |J_1|\leq \frac{1}{\delta}\int_{A_1} \frac{1}{m-n} \,dn = \frac{1}{\delta} \ln\Big( \frac{m}{a_1 + 1/(2\sqrt{\tau}) -m } \Big).\]
		 Since $ m\leq a_1 $, it then follows from the above estimate that $ |J_1| \ls  1+\ln\tau$. 
		 
		For $ J_2 $, we use the property $ f(a_1) = 0 $ to rewrite  		
	$J_2 = \int_{A_2} \frac{1}{m+n}\, \frac{f(n)-f(a_1)}{m-n} \,dn$.
		In addition, due to the assumption $ m \leq a_1 $, we have 
		$ 0< a_1-n \leq 3(m-n), \quad \forall\, n \leq 2m-a_1$. 
		Moreover, $ \frac{1}{m+n}\leq \frac{1}{\delta} $. Consequently,  according to \eqref{f-mv}, one has
		\[ |J_2| \leq 3B\sqrt{\tau} \int_{A_2} \frac{1}{m+n}\,dn \leq 3B\sqrt{\tau} \frac{1}{2\delta\sqrt{\tau}} \ls 1.\]
	
	For $ J_3 $, since $ [2m-a_1,a_1] $ is centered at $ m $, we estimate it as 
	\[\begin{split} 
	|J_3| &\leq  \bigg| \int_{A_3} \frac{1}{m^2-n^2}\, \big(f(n)-f(m)\big) \,dn \bigg| + f(m) \bigg| \int_{A_3} \frac{1}{m^2-n^2} \,dn \bigg|:= J_{31} + J_{32}.
	\end{split}\]
	Since $ |f(n) - f(m)| \leq B |m-n|\sqrt{\tau} $ and $ m\geq a_1 - \frac{1}{4\sqrt{\tau}} >\delta $,
	\[ |J_{31}| \leq B\sqrt{\tau} \int_{A_3} \frac{1}{m+n} \,dn \leq \frac{B}{\delta}\sqrt{\tau}(2a_1-2m) \ls 1. \]
	While for the term  $ J_{32} $,  due to the fact $ \int_{A_3} \frac{1}{m-n} \,dn = 0 $, one has
	\[|J_{32}|\ls \l|\int_{A_3} \frac{1}{m+n} \,dn\r| + \l|\int_{A_3} \frac{1}{m-n} \,dn\r| \ls 1.\]
	
		
		\item {\bf Case 2.3: $ \frac{\delta}{2}\leq m \leq a_1 -  \frac{1}{4\sqrt{\tau}} $.}
		
		For convenience of notation, we denote $ \lambda = \frac{\delta}{4\sqrt{\tau}} $ so that $ \lambda \leq \min\big\{ \frac{1}{8\sqrt{\tau}}, \, \frac{\delta}{4} \big\}<m$. Again, we split the interval $ [0,a_1] $ into three pieces 
		\[ [0, a_1] = [0, m-\ld] \cup [m-\ld, m+\ld] \cup [m+\ld, a_1] \]
and write 
		$ F_{31}(m,\tau) = \bigg( \int_{0}^{m-\ld} + \int_{m-\ld}^{m+\ld} + \int_{m+\ld}^{a_1}\bigg) \frac{1}{m^2-n^2} \, f(n) \,dn := K_1 + K_2 + K_3.$
		
		For $ K_1 $, by noticing $ m+n\geq m\geq \frac{\delta}{2} $, we obtain
		\[ |K_1| \leq \frac{2}{\delta} \int_{0}^{m-\ld} \frac{1}{m-n} \,dn = \frac{2}{\delta} \ln\Big( \frac{m}{\ld}\Big) \ls 1+\ln\tau. \]
		In the same way, we deduce $ |K_3|\ls 1+\ln\tau $. For $ K_2 $, we use the similar decomposition as that for $ J_3 $ in Case 2.2 to get 
		\[\begin{split} 
			|K_2| &\leq  \bigg| \int_{m-\ld}^{m+\ld} \frac{1}{m^2-n^2}\, \big(f(n)-f(m)\big) \,dn \bigg| + f(m) \bigg| \int_{m-\ld}^{m+\ld} \frac{1}{m^2-n^2} \,dn \bigg| := K_{21} + K_{22}.
		\end{split}\]
		Then analogous to the treatment of $ J_{31} $ and $ J_{32} $, we find $ |K_{21}| + |K_{22}| \ls 1 $, which implies $ |K_{2}| \ls 1$. As a result, $ |F_{31}(m,\tau)| \ls 1+\ln\tau $.	
		
		\item {\bf Case 2.4: $ 0< m\leq \frac{\delta}{2} $.}
		
		In this case, we split the integral into two parts:
		\[ F_{31}(m, \tau) = \bigg(\int_{0}^{\delta} + \int_{\delta}^{a_1} \bigg) \frac{1}{m^2-n^2} \, f(n) \,dn := S_1 + S_2. \]
		In the first integral, $ f(n)=1 $ since $ n\leq \delta $. In addition, recalling the fact that $ \int_0^{\infty} \frac{1}{m^2-n^2} \,dn=0 $, so we know 
		\[
		|S_1| = \bigg| \int_{0}^{\delta} \frac{1}{m^2-n^2} \,dn \bigg| = \bigg| \int_{\delta}^{\infty} \frac{1}{m^2-n^2} \,dn \bigg|.
		\]
		Due to the assumption $ m\leq \frac{\delta}{2} $, we infer from the above relation that 
		\[ |S_1| = \int_{\delta}^{\infty} \frac{1}{n^2-m^2} \,dn \leq \frac43  \int_{\delta}^{\infty} \frac{1}{n^2} \,dn = C. \]
		On the other hand, it is also easy seen that $ |S_2| \ls 1 $. As a consequence, $ |F_{31}(m,\tau)| \ls 1 $.
		
	\end{itemize}
The combination of the above Case 2.1--Case 2.4 shows that 
	\be\label{F31-est}
	F_{31}(m,\tau) \ls 1+\ln\tau, \quad\forall\, 0<m\leq 2, \,\tau\geq 1.
	\ee
%

	Next, we turn to consider $ F_{41} $ which is given in (\ref{F3-F4}) by splitting the integral interval as, 
	\[ F_{41}(m,\tau) = \bigg( \int_{\delta}^{4} + \int_{4}^{\infty} \bigg) \frac{1}{m^2-n^2} \, \big(1-f(n) \big) \,dn. \]
	Since $ m\leq 2 $, it can be readily checked that:
	\[ \bigg| \int_{4}^{\infty} \frac{1}{m^2-n^2} \, \big(1-f(n) \big) \,dn \bigg| \ls \int_{4}^{\infty} \frac{1}{n^2-4} \,dn \ls 1. \]
	For the first integral $  \int_{\delta}^{4} \frac{1}{m^2-n^2} \, \big(1-f(n) \big) \,dn $, we can apply similar argument as that for $ F_{31} $  by considering three subcases: (1) $ 0<m\leq \delta $; (2) $ \delta\leq m \leq \delta + \frac{\delta}{2\sqrt{\tau}} $; (3) $ \delta + \frac{\delta}{2\sqrt{\tau}} \leq m \leq 2 $. As a result, we can obtain 
	\be\label{F41-est}
	|F_{41}(m,\tau)| \ls 1+\ln\tau, \quad\forall\, 0<m\leq 2, \,\tau\geq 1.
	\ee
	Hence, substituting  estimates (\ref{F32-F42-est}), (\ref{F31-est}) and (\ref{F41-est}) into (\ref{F3-F4}) yields 
	\be\label{F-est-Case2} |F(m\sqrt{\tau}, \tau)| \leq |F_3(m,\tau)| + |F_{4}(m,\tau)| \ls \frac{m(1+\ln\tau)}{\sqrt{\tau}}, \quad\forall\, 0<m\leq 2, \,\tau\geq 1. \ee
\end{itemize}
	Finally, we complete the proof by combining estimate (\ref{F-est-Case1}) in Case 1 and estimate (\ref{F-est-Case2}) in Case 2. 
\end{proof}

With lemmas introduced above, we are ready to establish Proposition \ref{bdrx} and \ref{bdrh}

\begin{proof}[\bf Proof for Proposition \ref{bdrx}]
	According to formula (\ref{bdrop-f}) for the boundary integral operator $ \Phi_{bdr}(h) $, it suffices to show, for $j=1,2,3$,
\[
\|\th(t)I^*_j(x,t)\|_{X^{s,b}}\lesssim \|h\|_{H^{\frac{2s+1}{4}}(\R^+)}.
\]
For the case $j=1$, according to Lemma \ref{lemma1}, one has, for $s\in \R$ and $0<b<1$,
\begin{align*}
	\|\th(t)I^*_1(x,t)\|_{X^{s,b}}=&\l\|\th(t)\int_\R  e^{-i\rho^2 t}e^{i\rho x} \chi_{\R^+}(\rho)\rho \tilde{h}(-i\rho^2) d\rho\r\|_{X^{s,b}}
	\lesssim \l\|\chi_{\R^+}(\rho)\rho \wh{h_*}(\rho^2)\r\|_{H_\rho^s(\R)}= \|h\|_{H^{\frac{2s+1}{4}}(\R^+)}.
\end{align*}
For the case $j=2$, the extension $I_2^*(x,t)$ is in fact a $C^\infty$ function for both $x$ and $t$ with, 
\[\|I_2^*(x,t)\|_{L^2(
	\R^2)} \ls \|\phi_1(\mu )\wh{h_*}(\mu)\|_{L^2_\mu(\R)} \ls \|h\|_{H^s(\R^+)},\]
for any $s\in \R $ since supp $\phi_1\in (-1,2)$. Note that the estimate also holds for $\partial_x^j \partial_t^k I^*_2$, for any $j,k\geq 0$. It thus can be checked that, for $s+3b\geq0$ and $b\geq0$, 
 \[\|I^*_2(x,t)\|_{X^{s,b}}=\|\la \xi\ra^{s}\la \tau+\xi^2\ra^{b} \wh{I_2^*}(\xi,\tau)\|_{L^2(\R^2)}\ls \|\la \xi\ra^{s+2b}\la \tau\ra^{b} \wh{I_2^*}(\xi,\tau)\|_{L^2(\R^2)}\lesssim \|h\|_{H^{\frac{2s+1}{4}}(\R^+)}.\]
 Hence, according to Lemma \ref{xsb}, one has,
 \[\|\th(t)I^*_2(x,t)\|_{X^{s,b}}\lesssim \|h\|_{H^{\frac{2s+1}{4}}(\R^+)}.\]
While for $j=3$, combining Lemma \ref{xsb}--\ref{Lemma, I32} and set $b=\frac12-{\epsilon}$ for $0<\epsilon<\epsilon_0$, one can obtain, for $s< \frac32$,
\begin{equation*}
	\|\th(t)I^*_3(x,t)\|_{X^{s,b}}\lesssim \|h\|_{H^{\frac{2s+1}{4}}(\R^+)}.
\end{equation*}
\end{proof}

\begin{proof}[\bf Proof for Proposition \ref{bdrh}]
	It suffices to show, for $j=1,2,3$,
	\[
	\|I^*_j(x,t)\|_{H_x^s(\R)}\lesssim \|h\|_{H^{\frac{2s+1}{4}}(\R^+)}.
	\]
	For the case $j=1$, one has
	\begin{align*}
	\|I^*_1(x,t)\|_{H_x^s(\R)}=\l\|\int_\R  e^{-i\rho^2 t}e^{i\rho x} \chi_{\R^+}(\rho)\rho \tilde{h}(-i\rho^2) d\rho\r\|_{H_x^s(\R)}
	\lesssim \l\|\chi_{\R^+}(\rho)\rho \tilde{h}(-i\rho^2)\r\|_{H^s(\R)}
	= \|h\|_{H^{\frac{2s+1}{4}}(\R^+)}.
	\end{align*}
	For the case $j=2$, it follows directly from the argument as in Proposition \ref{bdrx} for $j=2$.\\
	\noindent
	For the case $j=3$, we write
	\begin{align*}
			\widehat{I^*_3}^x(\xi,t):=G_1(\xi,t)-\frac{i}{2\pi}G_2(\xi,t)=&\int_\R e^{it\tau}\frac{\sqrt{\tau}}{\xi^2+\tau}\phi_2(\tau)\hat{h}(\tau)(1-\Theta(\xi,\tau))(1+w(\tau))d\tau\\
			&-\frac{i}{2\pi} \int_\R e^{it\tau} \phi_2(\tau)\hat{h} (\tau)\int_\R \frac{1}{\xi-\eta}\frac{\sqrt{\tau}}{\eta^2+\tau} \Theta_1(\eta, \tau)d\eta d\tau
	\end{align*}
	with $\Theta_1$ defined in \eqref{Th1}. 
	\begin{cla}
		For $-\frac32<s<\frac12$, one has
		\[\sup_{t\in \R} \|\la \xi\ra^s G_j(\xi,t)\|_{L^2(\R)} \lesssim \|h\|_{H^{\frac{2s+1}{4}}(\R^+)}\]
		with $j=1,2$.
	\end{cla}
Once the claim is established, we can easily check that
\[	\|I^*_3(x,t)\|_{H_x^s(\R)}\lesssim \|h\|_{H^{\frac{2s+1}{4}}(\R^+)}.\]
To verify the claim, it suffices to consider the case for $\xi>0$ and the case for $\xi<0$ will be similar, we thus place discussion as follows.
\begin{itemize}
	\item{\bf For the term  $G_1$.} We first write
	\begin{align}
		\int_{\R^+} \la \xi\ra^{2s} G_1^2(\xi,t)d\xi=& \int_{\R^+}\la \xi\ra^{2s} \int_{\R^+} e^{it\tau_1}\phi_2(\tau_1)\frac{\sqrt{\tau_1}}{\xi^2+\tau_1}\hat{h}(\tau_1)(1-\Theta(\xi,\tau_1))(1+w(\tau_1))d\tau_1\nonumber\\
		&\times \int_{\R^+} e^{it\tau_2}\phi_2(\tau_2)\frac{\sqrt{\tau_2}}{\xi^2+\tau_2}\hat{h}(\tau_2)(1-\Theta(\xi,\tau_2))(1+w(\tau_2))d\tau_2d\xi\nonumber\\
		\lesssim &\int_{\R^+}\int_{\R^+} \tau_1^{\frac12}\tau_2^{\frac12}\phi_2(\tau_1)\phi_2(\tau_2) G_{11}(\tau_1,\tau_2) \hat{h}(\tau_1)\hat{h}(\tau_2) d\tau_1d\tau_2 \label{G11}
	\end{align}
	with \[G_{11}(\tau_1,\tau_2) =\int_{\R^+}  \frac{\la \xi\ra^{2s}}{(\xi^2+\tau_1)(\xi^2+\tau_2)}|(1-\Theta(\xi,\tau_1))(1-\Theta(\xi,\tau_2))|d\xi. \]
	Without loss of generality, we consider $1\leq \tau_1 \leq \tau_2$. Note that $1-\Theta(\xi,\tau_2)\equiv 0$ if $\xi<\delta \sqrt{\tau_2}$. We then denote $\delta_1:=\max{\{\d \sqrt{\tau_2}, \sqrt{\tau_1} \}}$
	\begin{align*}
		G_{11}(\tau_1,\tau_2)=& \l( \int^{\delta_1
	}_ {\d \sqrt{\tau_2}}+ \int_{\delta_1}^ {2 \sqrt{\tau_2}}+ \int_{2 \sqrt{\tau_2}}^\infty\r)  \frac{\la \xi\ra^{2s}}{(\xi^2+\tau_1)(\xi^2+\tau_2)}|(1-\Theta(\xi,\tau_1))(1-\Theta(\xi,\tau_2))|d\xi\\
		=& G_{11-1}(\tau_1,\tau_2)+G_{11-2}(\tau_1,\tau_2)+G_{11-3}(\tau_1,\tau_2).
	\end{align*}
	For $G_{11-1}$, we only need to consider 
	 $\delta_1=\sqrt{\tau_1}$,
	 otherwise the integral is 0. Then, 
	 \begin{align*}
	 	G_{11-1}(\tau_1,\tau_2)\lesssim\int^{\sqrt{\tau_1}}_{\d \sqrt{\tau_2}}  \frac{\la \xi\ra^{2s}}{(\xi^2+\tau_1)(\xi^2+\tau_2)} d\xi
	 	\lesssim  \int^{\sqrt{\tau_2}}_{\d \sqrt{\tau_2}} \la \xi\ra^{2s-4}d\xi\lesssim \tau_2^{\frac{2s-3}{2}}.
	 \end{align*}
	 	One can then also ensure that  
	 	\[G_{11-2}(\tau_1,\tau_2)	\lesssim  \tau_2^{\frac{2s-3}{2}}, \quad G_{11-3}(\tau_1,\tau_2)	\lesssim  \tau_2^{\frac{2s-3}{2}},\]
	 			if  $s<\frac32$.
%
Hence, these lead to 
	 $G_{11}(\tau_1,\tau_2)\lesssim\tau_2^{\frac{2s-3}{2}}$. 
	 
	 Therefore, it follows  that,  if $\frac{2s-1}{4}<\frac12$, that is $s<\frac32$ (See. Page 245 in \cite{HLG}), one has
	 \begin{equation}\label{Hardy}
	 	 \begin{aligned}
	 	\mbox{\mbox{RHS of }\eqref{G11}}\lesssim &\int_{\R^+} \int^{\tau_2}_0 \tau_1^{\frac12} \tau_2^{\frac12} \tau_2^{\frac{2s-3}{2}} \hat{h}(\tau_1)\hat{h}(\tau_2)d\tau_1d\tau_2\\
	 	\lesssim & \int_{\R^+} \int^{\tau_2}_0 \tau_1^{\frac{1-2s}{4}}\tau_2^{\frac{2s-5}{4}} \tau_1^{\frac{2s+1}{4}}\tau_2^{\frac{2s+1}{4}}\hat{h}(\tau_1)\hat{h}(\tau_2)d\tau_1d\tau_2\\
	 	\lesssim& \int_{\R^+} \l(\tau_2^{\frac{2s+1}{4}}\hat{h}(\tau_2)\r)\tau_2^{\frac{2s-5}{4}}\int^{\tau_2}_0 \tau_1^{\frac{1-2s}{4}} \l(\tau_1^{\frac{2s+1}{4}}\hat{h}(\tau_1)\r)d\tau_1d \tau_2\\
	 	\lesssim &\|h\|_{H^{\frac{2s+1}{4}}(\R^+)} \l(\int_{\R^+} \tau_2^{\frac{2s-5}{2}} \l(\int^{\tau_2}_0 \tau_1^{\frac{1-2s}{4}} \l(\tau_1^{\frac{2s+1}{4}}\hat{h}(\tau_1)\r) d\tau_1 \r)^2d\tau_2\r)^{\frac12}\\
	 	\lesssim &\|h\|^2_{H^{\frac{2s+1}{4}}(\R^+)},
	 	\end{aligned}
	 \end{equation}
	 
	 \item{\bf For the term $G_2$. } Similar to $G_1$, we consider
	 \begin{align*}
	 	\int_{\R^+} \la \xi\ra^{2s} G^2_2(\xi,t)d\xi = &	\int_{\R^+} \la \xi\ra^{2s}\l(\int_{^+} e^{it\tau_1} \phi_2(\tau_{1})\hat{h} (\tau_1)\int_\R \frac{1}{\xi-\eta}\frac{\sqrt{\tau_1}}{\eta^2+\tau_1} \Theta_1(\eta, \tau_1)d\eta d\tau_1\r)\\
	 	&\times\l(\int_{^+} e^{it\tau_2}\phi_2(\tau_{2}) \hat{h} (\tau_2)\int_\R \frac{1}{\xi-\eta}\frac{\sqrt{\tau_2}}{\eta^2+\tau_2} \Theta_1(\eta, \tau_2)d\eta d\tau_2\r)d\xi\\
	 	=& \int_{\R^+} \int_{\R^+} G_{21}(\tau_1, \tau_2)\phi_2(\tau_{1})\phi_2(\tau_2)  \hat{h} (\tau_1) \hat{h} (\tau_2)d\tau_1d\tau_2,
	 \end{align*}
	where
	 $G_{21}(\tau_1, \tau_2)=\int_{\R^+} \la \xi\ra^{2s}  F(\xi,\tau_1)F(\xi,\tau_2)d\xi,$
	 and $F(\xi,\tau)$ is defined in \eqref{F}. 
	 Again, we only need to consider  the case for $1\leq \tau_1\leq \tau_2$. Therefore, we can recall   Lemma \ref{Lemma, F-est} and compute $G_{21}$ as follows.
	 \begin{align*}
	 		 G_{21}(\tau_1,\tau_2)=& \Bigg(\int^{2\sqrt{\tau_1}}_0 +\int^{2\sqrt{\tau_2}}_ {2\sqrt{\tau_1}} + \int^{\infty}_{2\sqrt{\tau_2}}\Bigg)\la \xi\ra^{2s}  F(\xi,\tau_1)F(\xi,\tau_2)d\xi\\
	 		 :=& G_{21-1}+G_{21-2}+G_{21-3}.
	 \end{align*}

   For $G_{21-1}$, by setting $m=\xi /\sqrt{\tau_1}$, we have 
   \begin{align*}
   	G_{21-1}(\tau_1,\tau_2)\lesssim &\int^{2 \sqrt{\tau_1}}_0 \la \xi\ra^{2s}  \tau_1^{-1} \xi \l( 1+ \ln \tau_1\r)\tau_2^{-1} \xi \l( 1+ \ln \tau_2\r)d\xi\\
   	\lesssim& \tau_1^{-\frac12+} \tau_2^{-1+} \int^{a_0}_0 \la m\sqrt{\tau_1}\ra^{2s+2}dm 
   	\lesssim  \tau_1^{s+\frac12+}\tau_2^{-1+},
   \end{align*}
   if $2s+2>-1$, that is, $s>-\frac32$.
   
   For $G_{21-2}$, we have, for $s<\frac12$, 
  \begin{align*}
  		G_{21-2}(\tau_1,\tau_2) \lesssim & \int^{2\sqrt{\tau_2}}_ {2\sqrt{\tau_1}} \la \xi\ra^{2s} \frac{\tau_1}{\xi^{-3}} \frac{\xi(1+\ln \tau_2 )}{\tau_2}  d\xi
  		\lesssim  \tau_1 \tau_2^{-1+} \int^{\infty}_{2\sqrt{\tau_1}}  \la \xi\ra^{2s-2} d\xi \lesssim  \tau_1^{s+\frac12}\tau_2^{-1+}.
  \end{align*}
    For $G_{21-3}$, we have,   
    \begin{align*}
    G_{21-3}(\tau_1,\tau_2) \lesssim & \int_{2\sqrt{\tau_2}}^{\infty} \la \xi\ra^{2s} \tau_1   \tau_2  \xi^{-6} d\xi
    \lesssim \tau_1 \tau_2 \int^\infty_{2\sqrt{\tau_2}} \la \xi\ra^{2s-6}d\xi 
    \ls \tau_1 \tau_2 \tau_2^{s-\frac52}\leq \tau_1^{s+\frac12} \tau_2^{-1}.
    \end{align*}
       if $2s-6<-1$ and $s-1/2<0$, that is, $s<\frac12$. Therefore, it follows from above estimates that 
      \begin{align*}
      	\int_{\R^+} \la \xi\ra^{2s} G^2_2(\xi,t)d\xi \lesssim& \int_{\R^+} \int^{\tau_2}_0  \tau_1^{s+\frac12+}\tau_2^{-1+} \hat{h}(\tau_1)\hat{h}(\tau_2)d\tau_1d\tau_2\\
      	\lesssim & \int_{\R^+} \int^{\tau_2}_0 \tau_1^{\frac{2s+1}{4}+} \tau_2^{-1-(\frac{2s+1}{4}+)}\l(\tau_1^{\frac{2s+1}{4}}\tau_2^{\frac{2s+1}{4}} \hat{h}(\tau_1)\hat{h}(\tau_2)\r)d\tau_1d\tau_2\\
      	\lesssim& \|h\|^2_{H^{\frac{2s+1}{4}}(\R^+)}.
      \end{align*}
      Similar to the discussion in \eqref{Hardy}. The last inequality holds if $-\frac{2s+1}{4}<\frac12$, that is, $s>-\frac32$.
\end{itemize}
 The proof is now complete.
\end{proof}

\section{Preliminary}
In this section, we introduce some important inequalities and existing results for the Schr\"odinger equation that will be used in later study. Lemmas below can be found in \cite{ET16,XZ-1}.
\begin{lem}\label{lem1}
	Let $\rho_{1}$, $\rho_{2}<1$ and $\rho_{1}+\rho_{2}>1$. Then there exists  a constant $C=C(\rho_{1},\rho_{2})$ such that for any $c_1,c_2\in\m{R}$,
	\be\label{int in tau}
	\int_{-\infty}^{\infty}\frac{dx}{\la x-c_1 \ra^{\rho_{1}} \la x-c_2 \ra^{\rho_{2}}}\leq \frac{C}{\la c_1-c_2\ra^{\rho_{1}+\rho_{2}-1}}.\ee
	In addition, if $\rho_1\geq\rho_2>1$, then one has
		\be\label{int in tau-1}
	\int_{-\infty}^{\infty}\frac{dx}{\la x-c_1 \ra^{\rho_{1}} \la x-c_2 \ra^{\rho_{2}}}\leq \frac{C}{\la c_1-c_2\ra^{\rho_{2}}}.\ee
\end{lem}

\begin{lem}\label{lem2}
	If $\rho>\frac{1}{2}$, then there exists $C=C(\rho)$ such that for any $c_{i}\in\m{R},\,0\leq i\leq 2$, with $c_{2}\neq 0$,
	\be\label{bdd int for quad}
	\int_{-\infty}^{\infty}\frac{dx}{\la c_{2}x^{2}+c_{1}x+c_{0}\ra^{\rho}}\leq \frac{C}{|c_{2}|^{1/2}\la c_0-\frac{c_1^2}{4c_2}\ra^{1/2}}.\ee
	\end{lem}

Consider the following problem
\begin{equation*}
\begin{cases}
iu_t+u_{xx}=0, x,t\in\R\\
u(x,0)=\varphi(x).
\end{cases}
\end{equation*}
We denote the solution $u$ as $u=W_{R}(\varphi)$. Then, one has 
\[[W_{R}(\varphi)](x,t)=\int_\R e^{-i\xi^2t}e^{i\xi x} \hat{\varphi}(\xi)d\xi,\]
The following result for the linear IBVP of Schr\"odinger equation comes from \cite{BSZ-1}.

\begin{pro}\label{represent}
	The solution of the IBVP, 
	\begin{equation*}
	\begin{cases}
	iu_t+u_{xx}=f,\quad  x,t\in\R^+\\
	u(x,0)=\varphi(x),u(0,t)=h(t).
	\end{cases}
	\end{equation*}
	 can be written as
	\[u(x,t)= W_R(\varphi^*)+ W_{bdr}(h-p)
	+ 
	\left(\int^t_0 [W_R(f)](x,t-t')dt'-W_{bdr}(q)\right),\]
	where $\varphi^*$ is the zero extension of $\varphi$ from $\R^+$ to $\R$, $p(t)= W_R(\varphi^*)(0,t)$, and $q(t)=\int^t_0 [W_R(f)](0,t-t')dt'$.
\end{pro}

The following Lemma comes directly from \cite{51}.

\begin{lem}\label{lemma1}
For any $s$, $0<b<1$ and $0<\si<\frac12$, one has
\begin{equation}\label{R1}
\|\th(t) W_R(\varphi)\|_{H_x^s(\R)\cap X^{s,b}}\lesssim \|\varphi\|_{H^s(\R)}, \quad \| W_R(\varphi)(x)\|_{H_t^{\frac{2s+1}{4}}(\R)}\lesssim \| \varphi\|_{H^s(\R)}.
\end{equation}
\be\label{R4}
\l\|\th(t) \int W_R(f)(x,t-t')dt'\r\|_{X^{s,b}}\lesssim \| f\|_{X^{s,\si-1}}.
\ee
\end{lem}


Lemma below can be found in \cite{M16}.
\begin{lem}\label{kato}
For $s\leq 0$ and $\si>\frac12$, one has, 
\begin{equation}\label{f2}
\left\|\th(t)\int^t_0  W_R(t-t')(f)(x,t')dt' \right\|_{H_t^{\frac{2s+1}{4}}(\R)}\lesssim \| f\|_{X^{s,\si-1}}+\|f\|_{Z^{s,\si-1}}.
\end{equation}
\end{lem}

\section{Bilinear estimate}

Though the following theorem has been established in \cite{M16}, we provide a slight different approach based on our early work for coupled KdV and Boussiensq type equations. 
\begin{thm}\label{bil}
For $-\frac34<s<0$ and $\max\{\frac38,\frac18-\frac{s}{2}\}<b<\frac12$, there exists a $\si_0>\frac12$ such that for any $\frac12<\si<\si_0$,
\be\label{bi-1}\|uv\|_{X^{s,\si-1}}\lesssim \|u\|_{X^{s,b}}\|v\|_{X^{s,b}}.\ee
\be\label{bi-2}\|uv\|_{Z^{s,\si-1}}\lesssim \|u\|_{X^{s,b}}\|v\|_{X^{s,b}}.\ee
\end{thm}
\begin{proof}

We first establish the estimate \eqref{bi-1}. Let $u, v\in X^{s,b }$ and $w\in X^{-s,1-\si}$, we  write
\[f_1(\xi,\tau)=\la \xi\ra^s \la\tau+\xi^2 \ra ^{b} \hat{u}(\xi,\tau), \quad f_2(\xi,\tau)=\la \xi\ra^s \la \tau+\xi^2 \ra ^{b}  \hat{v}(\xi,\tau),\quad f_3(-\xi,-\tau)=\la \xi\ra^{-s} \la \tau+\xi^2 \ra^{1-\si} \hat{w}(\xi,\tau).\]
The idea of the proof is inherited from \cite{tao2,tao1}, it is beneficial to reduce it to an estimate of some weighted convolution of $L^{2}$ functions.  For the convenience of notations, we denote $\vec{\xi}=(\xi_{1},\xi_{2},\xi_{3})$, $\vec{\tau}=(\tau_{1},\tau_{2},\tau_{3})$ and
\be\label{int domain}
A:=\Big\{(\vec{\xi},\vec{\tau})\in\m{R}^{6}:\sum_{i=1}^{3}\xi_{i}=\sum_{i=1}^{3}\tau_{i}=0\Big\}.\ee
Then, according to duality and Plancherel theorem, we set $\rho=-s\geq 0$, \eqref{bi-1} becomes,
\be\label{bi}
\int_A \frac{ \la\xi_1\ra^{\rho}\la\xi_2\ra^{\rho}f_1(\xi_1,\tau_1)f_2(\xi_2,\tau_2)f_3(\xi_3,\tau_3)}{  \la\xi_3\ra^{\rho }  \la L_1\ra^{b}\la L_2\ra^{b}\la L_3\ra^{1-\si}}\lesssim \prod^3_{j=1}\|f_j\|_{L^2_{\xi,\tau}},
\ee
where
\[L_1:= \tau_1+\xi_1^2, \quad L_2:= \tau_2+\xi_2^2, \quad L_3:= \tau_3-\xi_3^2, \]
and
\[(\xi_2,\tau_2)=(\xi-\xi_1,\tau-\tau_1), \quad (\xi_3,\tau_3)=(-\xi,-\tau).\]
Notice that $\la \xi_1\ra \la \xi_2\ra /\la \xi_3\ra \geq 1$, we just need to consider the case when $\rho$ is close to $\frac34$. 
Without loss of generality, we assume $\frac{1}{2} < \rho<\frac34$ and define $\si_0= \frac{1}{3}(2b-\rho+\frac54)$ to help to justify the proof for $\frac12< \si\leq \si_0$.
In addition, we denote $$H:=L_1+L_2+L_3=\xi_1^2+\xi_2^2-\xi_3^2=-2\xi_1\xi_2,$$
thus, $|H|\sim |\xi_1\xi_2|$ which will be used repeatedly in the later proof. Furthermore,
if $(\xi_1,\tau_1)$ is fixed, then by substituting  $\xi_3=-(\xi_1+\xi_2)$ and $\tau_3=-(\tau_1+\tau_2)$, $L_2+L_3$ can be viewed as a function of $\xi_2$, and we denote
\begin{align*}
P_1(\xi_2):=L_2+L_3=\tau_2+\xi_2^2+\tau_3-\xi_3^2=-\tau_1+\xi^2_2-(\xi_1+\xi_2)^2=-L_1-2\xi_1\xi_2
\end{align*}
with
$P_1'(\xi_2)=-2\xi_1$. Similarly, we set
\begin{align*}
P_2(\xi_1):=L_1+L_2=\tau_1+\xi_1^2+\tau_2+\xi_2^2
=-\tau_3+\xi^2_1+(\xi_1+\xi_3)^2=-L_3+2\xi_1\xi_3+2\xi_1^2.
\end{align*}
Next, we will adapt the idea from our earlier work on the coupled KdV system \cite{LCYZ23} as well as the sixth order Boussinesq equation \cite{96-1}, and establish the estimate in the following cases.
\begin{itemize}
  \item{\bf Case 1.} If $|\xi_1|\leq 1$, it implies that $\la\xi_1\ra \la\xi_2\ra /\la\xi_3\ra\lesssim 1$, then \eqref{bi} reduces to
      \be\label{bi1}\int_A \frac{ |f_1(\xi_1,\tau_1)f_2(\xi_2,\tau_2)f_3(\xi_3,\tau_3)|}{  \la L_1\ra^{b}\la L_2\ra^{b}\la L_3\ra^{1-\si}}\lesssim \prod^3_{j=1}\|f_j\|_{L^2_{\xi,\tau}}.\ee
     We notice that
       \begin{align*}
        \mbox{LHS of \eqref{bi1}}\lesssim    & \iint \frac{|f_3|}{\la L_3\ra^{1-\si}} \l(\iint\frac{\chi_{|\xi_1|\leq 1}|f_1f_2|}{\la L_1\ra^{b} \la L_2\ra^{b}}d\xi_1 d\tau_1\r)d\xi_3 d\tau_3\\
  \lesssim & \iint \frac{|f_3|}{\la L_3\ra^{1-\si}} \l(\iint\frac{\chi_{|\xi_1|\leq 1}d\xi_1 d\tau_1}{\la L_1\ra^{2b}  \la L_2\ra^{2b}}\r)^\frac12\l(\iint |f_1f_2|^2 d\xi_1 d\tau_1\r)^\frac12d\xi_3 d\tau_3.
      \end{align*}
      According to Lemma \ref{lem1} and $\tau_1+\tau_2+\tau_3=0$, we  can obtain that, for $b>\frac38$, $|\xi_1|\leq 1$ and $\frac12<\si<1$,
  \[\sup_{\xi_3,\tau_3} \frac{1}{\la L_3\ra^{2-2\si}}\int_{|\xi_1|\leq 1} \frac{d\xi_1}{\la L_1 +L_2\ra^{4b-1}}\lesssim 1.\]
  Therefore, \eqref{bi1} is achieved through Cauchy-Schwartz inequality,
  \begin{align*}
     \mbox{LHS of \eqref{bi1}}\lesssim    \iint |f_3| \l(\iint f_1^2f_2^2 d\xi_1 d\tau_1\r)^\frac12d\xi_3 d\tau_3
     \lesssim  \prod^3_{j=1}\|f_j\|_{L^2_{\xi_\tau}}.
  \end{align*}
  \item{\bf Case 2.} If $|\xi_2|\leq 1$, it is similar to the \textbf{Case 1}, therefore omitted.
  \item{\bf Case 3.} If $|\xi_1|, |\xi_2|>1$, and $|L_1|=MAX:=\max\{|L_1|,|L_2|,|L_3|\}$, then, for $b>\frac38$ and $\frac12<\si<1$, one has,
     \[\la L_1\ra^b \la L_2\ra^b \la L_3\ra^{1-\si}\geq \la L_1\ra^{2b-3\si+1} \la L_2\ra^\si \la L_3\ra^{\si}.\]
     In addition, we point out that $$MAX\geq \frac13 |H| \gtrsim |\xi_1\xi_2|.$$
     To establish \eqref{bi}, we first establish that 
     \begin{align*}
&\int_A \frac{ \la\xi_1\ra^{\rho}\la\xi_2\ra^{\rho}f_1(\xi_1,\tau_1)f_2(\xi_2,\tau_2)f_3(\xi_3,\tau_3)}{  \la\xi_3\ra^{\rho }  \la L_1\ra^{b}\la L_2\ra^{b}\la L_3\ra^{1-\si}}
\lesssim \int_A \frac{ \la\xi_1\ra^{\rho}\la\xi_2\ra^{\rho}f_1(\xi_1,\tau_1)f_2(\xi_2,\tau_2)f_3(\xi_3,\tau_3)}{  \la\xi_3\ra^{\rho }  \la L_1\ra^{2b-3\si+1}\la L_2\ra^{\si}\la L_3\ra^{\si}},\\
  \lesssim & \iint \frac{\la\xi_1\ra^{\rho}|f_1|}{\la L_1\ra^{2b-3\si+1}} \l(\iint\frac{\la\xi_2\ra^{2\rho}}{\la \xi_3\ra^{2\rho}\la L_2\ra^{2\si}  \la L_3\ra^{2\si}}d\xi_2 d\tau_2\r)^\frac12\l(\iint f_2^2f_3^2 d\xi_2 d\tau_2\r)^\frac12d\xi_1 d\tau_1,
\end{align*}
it then suffices to bound the term
\be\label{bi2}
\sup_{\xi_1,\tau_1}\frac{\la\xi_1\ra^{2\rho}}{\la L_1\ra^{4b-6\si+2}} \int\frac{\la\xi_2\ra^{2\rho}\la\xi_3\ra^{-2\rho}}{\la L_2+ L_3\ra^{2\si}}d\xi_2.
\ee
\begin{itemize}
  \item{\bf Case 3.1.} $|\xi_1|\ll |\xi_3|$.\\
  In this case, one has $|\xi_2|\sim |\xi_3|$. Then, it follows from   \eqref{bi2}  that
  \begin{align}
     \frac{\la\xi_1\ra^{2\rho}}{\la L_1\ra^{4b-6\si+2}} \int\frac{\la\xi_2\ra^{2\rho}\la\xi_3\ra^{-2\rho}}{\la L_2+ L_3\ra^{2\si}}d\xi_2\lesssim &  \frac{\la\xi_1\ra^{2\rho}}{\la L_1\ra^{4b-6\si+2}} \int \frac{1}{|\xi_1|} \frac{|P_1'(\xi_2)|}{\la P_1(\xi_2)\ra^{2\si}}d\xi_2\lesssim \frac{|\xi_1|^{2\rho-1}}{\la L_1\ra^{4b-6\si+2}}.\label{p11}
  \end{align}
  Since
  $|\xi_1|^{2\rho-1}\lesssim |\xi_1\xi_2|^{\frac{2\rho-1}{2}}\lesssim | L_1|^{\frac{2\rho-1}{2}}\lesssim \la L_1\ra^{\frac{2\rho-1}{2}},$
  then \eqref{p11} is bounded if $\rho<4b-6\si+\frac52$.
  \item{\bf Case 3.2.} $|\xi_1|\sim |\xi_3|$.\\
  In this case, it yields that $|\xi_2|\lesssim |\xi_1|\sim |\xi_3|$. This leads \eqref{bi2} to
    \begin{align}
     \frac{\la\xi_1\ra^{2\rho}}{\la L_1\ra^{4b-6\si+2}} \int\frac{\la\xi_2\ra^{2\rho}\la\xi_3\ra^{-2\rho}}{\la L_2+ L_3\ra^{2\si}}d\xi_2\lesssim &  \frac{1}{\la L_1\ra^{4b-6\si+2}} \int \frac{|\xi_2|^{2\rho}}{|\xi_1|} \frac{|P_1'(\xi_2)|}{\la P_1(\xi_2)\ra^{2\si}}d\xi_2 .\label{p12}
  \end{align}
  Since
  $|\xi_1|^{-1}|\xi_2|^{2\rho}= |\xi_1|^{-1}|\xi_2|^{\frac{2\rho+1}{2}}|\xi_2|^{\frac{2\rho-1}{2}}
\lesssim |\xi_1\xi_2|^{\frac{2\rho-1}{2}}\lesssim   \la L_1\ra^{\frac{2\rho-1}{2}}, $
  then \eqref{bi2} is also bounded if $\rho<4b-6\si+\frac52$.
  \item{\bf Case 3.3.} $|\xi_1|\gg |\xi_3|$.
  In this case, we have $|\xi_1|\sim |\xi_2|$. Thus, \eqref{bi2} is bounded by
     \begin{align}
     \frac{\la\xi_1\ra^{2\rho}}{\la L_1\ra^{4b-6\si+2}} \int\frac{\la\xi_2\ra^{2\rho}\la\xi_3\ra^{-2\rho}}{\la L_2+ L_3\ra^{2\si}}d\xi_2\lesssim &  \frac{1}{\la L_1\ra^{4b-6\si+2}} \int \frac{|\xi_1\xi_2|^{2\rho}}{|\xi_1|} \frac{|P_1'(\xi_2)|}{\la P_1(\xi_2)\ra^{2\si}}d\xi_2 .\label{p13}
  \end{align}
  Since
  $|\xi_1|^{2\rho-1}|\xi_2|^{2\rho}\sim |\xi_1\xi_2|^{2\rho-\frac12}\lesssim   \la L_1\ra^{2\rho-\frac{1}{2}},$
  then \eqref{p13} is finite if $\rho< 2b-3\si+\frac54$.
\end{itemize}
  \item{\bf Case 4.}  If $|\xi_1|, |\xi_2|>1$, and $|L_2|=MAX $. Due to the symmetric structure between $\xi_1$ and $\xi_2$ in \eqref{bi}, the proof for this case is same as \textbf{Case 3}.
  \item{\bf Case 5.} If $|\xi_1|, |\xi_2|>1$, and $|L_3|=MAX $.
  
   In this case ,  we have
         $\la L_1\ra^b \la L_2\ra^b \la L_3\ra^{1-\si}\geq \la L_1\ra^\si \la L_2\ra^{\si}\la L_3\ra^{2b-3\si+1}.$
         Similar to the statements in \textbf{Case 3}, it suffices to bound the term
     \be\label{bi3}
\sup_{\xi_3,\tau_3}\frac{\la\xi_3\ra^{-2\rho}}{\la L_3\ra^{4b-6\si+2}} \int\frac{\la\xi_1\ra^{2\rho}\la\xi_2\ra^{2\rho}}{\la L_1+ L_2\ra^{2\si}}d\xi_1.
\ee
Recall that $P_2=L_1+L_2= -L_3+2\xi_1^2+2\xi_1\xi_3$, then, according to Lemma \ref{lem2}, one has
\begin{align*}
  \int \frac{d\xi_1}{\la L_1+ L_2\ra^{2\si}}=&  \int \frac{d\xi_1}{\la 2\xi_1^2+2\xi_1\xi_3-L_3\ra^{2\si}}
  \lesssim  \la L_3+\frac12 \xi_3^2\ra^{-\frac12}.
\end{align*}
Moreover, one has $\la\xi_1\ra^{2\rho}\la\xi_2\ra^{2\rho}\sim |\xi_1\xi_2|^{2\rho} \lesssim \la L_3 \ra^{2\rho}$. Hence, \eqref{bi3} is bounded by
\[\frac{\la L_3 \ra^{2\rho+6\si-4b-2}}{\la\xi_3\ra^{2\rho} \la L_3+\frac12 \xi_3^2\ra^{\frac12}}\lesssim 1,\]
which can be checked for $|L_3|\gg |\xi_3^2|$  if $\rho<2b-3\si+\frac54$ and $|L_3|\lesssim |\xi_3^2|$ if $\rho < 4b-6\si+2$.
\end{itemize}
We than move on to show estimate \eqref{bi-2}. Similar to the proof of \eqref{bi-1}, we set 
\[g_1(\xi,\tau)=\la \xi\ra^s \la\tau+\xi^2 \ra ^{b} \hat{u}(\xi,\tau), \quad g_2(\xi,\tau)=\la \xi\ra^s \la \tau+\xi^2 \ra ^{b}  \hat{v}(\xi,\tau),\quad g_3(-\xi,-\tau)=\la \tau\ra^{-s/2} \la \tau+\xi^2 \ra^{1-\si} \hat{w}(\xi,\tau).\]
It then suffices to prove 
\be\label{bi-3}
\int_A \frac{ \la\xi_1\ra^{\rho}\la\xi_2\ra^{\rho}|g_1(\xi_1,\tau_1)g_2(\xi_2,\tau_2)g_3(\xi_3,\tau_3)|}{  \la\tau_3\ra^{\rho/2 }  \la L_1\ra^{b}\la L_2\ra^{b}\la L_3\ra^{1-\si}}\lesssim \prod^3_{j=1}\|g_j\|_{L^2_{\xi,\tau}},
\ee
with $\tau_i$, $\xi_i$ and $L_i$, $i=1,2,3$, been defined as in \eqref{bi}.
\begin{itemize}
	 \item{\bf Case 1.} For $\la \xi_3^2\ra \lesssim \la \tau_3\ra $, one has $\la\tau_3\ra^{\rho/2 } \geq \la \xi_3 \ra^\rho$. We can establish \eqref{bi-3} as the proof for \eqref{bi}.
	 \item{\bf Case 2.} For $\la \xi_3^2\ra \gg \la \tau_3\ra $, one can deduce that
	 $|\xi_3|^2\gg |\tau_3|$ and $|\xi_3|\gg 1$, which also lead to $\la L_3\ra \sim \la \xi_3\ra^2$. We then can actually show 
	\be\label{bi-4} \int_A \frac{ \la\xi_1\ra^{\rho}\la\xi_2\ra^{\rho}|g_1(\xi_1,\tau_1)g_2(\xi_2,\tau_2)g_3(\xi_3,\tau_3)|}{    \la L_1\ra^{b}\la L_2\ra^{b}\la L_3\ra^{1-\si}}\lesssim \prod^3_{j=1}\|g_j\|_{L^2_{\xi,\tau}}.\ee
	 The proof is similar to  \eqref{bi}, we will only present a brief idea.
\begin{itemize}
	\item{\bf Case 2.1.} For $|\xi_1|\leq 1$ or $|\xi_2|\leq 1$ , according to  $\la L_3\ra^{\frac{\rho}{2}} \sim \la \xi_3\ra^\rho$, we can rewrite \eqref{bi-4} as
	\[\int_A \frac{ \la\xi_1\ra^{\rho}\la\xi_2\ra^{\rho}|g_1(\xi_1,\tau_1)g_2(\xi_2,\tau_2)g_3(\xi_3,\tau_3)|}{ \la \xi_3\ra^\rho   \la L_1\ra^{b}\la L_2\ra^{b}\la L_3\ra^{1-\si-\frac{\rho}{2}}}\lesssim \prod^3_{j=1}\|g_j\|_{L^2_{\xi,\tau}}.\]
	Then, since $\si<1-\frac{\rho}{2}$, the remainder proof can be established as the one in Case 1 of \eqref{bi}.
	\item{\bf Case 2.2.} For $|\xi_1|, |\xi|_2 \geq 1$, and $|L_1|=MAX $,  we follow the idea  in Case 2 of \eqref{bi}. It  then suffices to bound the term,
	\[ \frac{\la\xi_1\ra^{2\rho}}{\la L_1\ra^{4b-6\si+2}} \int\frac{\la\xi_2\ra^{2\rho} }{\la L_2+ L_3\ra^{2\si}}d\xi_2 \lesssim \frac{\la\xi_1\ra^{2\rho}}{\la L_1\ra^{4b-6\si+2}} \int \frac{\la\xi_2\ra^{2\rho}}{|\xi_1|} \frac{|P_1'(\xi_2)|}{\la P_1(\xi_2)\ra^{2\si}}d\xi_2.\]
	Recall that $\la\xi_1\ra^{2\rho-1}\la\xi_2\ra^{2\rho}\lesssim \la L_1\ra^{2\rho-\frac12}$, then the above bound is achieved for $\rho<2b-3\si+\frac54$.
	\item{\bf Case 2.3.} For $|\xi_1|, |\xi|_2 \geq 1$, and $|L_2|=MAX $,  we can just repeat the proof in Case 2.2.
	 	\item{\bf Case 2.4.}  For $|\xi_1|, |\xi|_2 \geq 1$, and $|L_3|=MAX $,  it again suffices to bound the term,
	 	\[\frac{ 1}{\la L_3\ra^{4b-6\si+2}} \int\frac{\la\xi_1\ra^{2\rho}\la\xi_2\ra^{2\rho}}{\la L_1+ L_2\ra^{2\si}}d\xi_1\lesssim \la L_3\ra^{2\rho-4b+6\si-2}\la L_3+\frac12 \xi_3^2\ra^{-\frac12}\lesssim \la \xi_3^2\ra^{2\rho-4b+6\si-\frac52},\]
	 	which can be also obtained for  $\rho<2b-3\si+\frac54$.
\end{itemize}
\end{itemize}
The proof is now complete.
\end{proof}

\section{Local well-posedness}

\begin{proof}[\bf Proof of Theorem \ref{alter-main}]
	Without loss of generality, we assume $T=1$. Let $\eps_0\in(0,1)$ be a constant which will be determined later.  For given 
	\[(\varphi,h)=Y_s:=\in  H^s(\R^+)  \times H^{\frac{2s+1}{4}}(\R^+),\]
	with
	\[
	\|\varphi\|_{H^s(\R^+)}+\|h\|_{\H^{\frac{2s+1}{4}}(\R^+)}\leq \eps_0,\] 
	and for $\max\{\frac38,\frac18-\frac{s}{2}\}<b<\frac12$, let
	\begin{equation*}
	S_{C,\eps_0}=\l\{u\in \X^{s,b}_+, \|u\|_{\X^{s,b}_+}\leq CE_0\r\}, \quad \mbox{where}\quad E_0:=\|\varphi\|_{H^s(\R^+)}+\|h\|_{\H^{\frac{2s+1}{4}}(\R^+)}\leq \eps_0.
	\end{equation*}
	Then the set $S_{C,\eps}$ is a convex, closed and bounded subset of $\X^{s,b}_+$.

	Given $u\in S_{C,\eps_0}$, we define
	\begin{align*}
	\Gamma(u)(x,t)=\eta(t)W_R(\varphi^*)+\eta(t)W_{bdr}(h-p)
	+\eta(t)
	\left(\int^t_0 [W_R(f)](x,t-t')dt'-W_{bdr}(q)\right),
	\end{align*}
	where $f= u^2$, $p$   and
	$q$  are as defined in Proposition \ref{represent}.
	It  can then  be verified
	that $\Gamma(u)(x,t)$ solves the IBVP in $[0,1]$.
	We  will show that $\Gamma$ is a contraction map from $S_{C,\eps_0}$ to $S_{C,\eps_0}$ for proper $C$ and $\eps_0$.  Applying Lemma \ref{lemma1} and  Proposition  \ref{bdrh}   leads to
	\begin{align*}
	&\|\Gamma(u)\|_{H^s(\R^+)}\\
	\lesssim &\|\eta(t)\l(W_R(\varphi^*)+W_{bdr}(h-p)\r)\|_{H_x^s(\R^+)}+\left\|\eta (t)
	\left(\int^t_0 [W_R(f)](x,t-t')dt'-W_{bdr}( q)\right)\right\|_{H_x^s(\R^+)}\\
	\lesssim & \|(\varphi, h)\|_{Y_s}+\left\|\eta(t)
	\left(\int^t_0 [W_R(f)](x,t-t)dt'\right)\right\|_{X_{+}^{s,\sigma_1}} + \left\|   \eta(t)W_{bdr}(q)  \right\|_{H^{s}_{x}(\m{R}^+)},
	\end{align*}
	recalling the fact $X^{s,\si_1}\subseteq C^0_t H^s_x$ for any $\si_1>\frac12$. 
	According to Lemma \ref{lemma1}, Proposition \ref{bdrx}, \ref{bdrh}, \ref{kato} and Theorem \ref{bil}, there exsits some $ \sigma_0=\sigma_0(s)\in(\frac12,1) $ such that for any $\si\in(\frac12,\si_0]$,
	\begin{align*}
	\left\|\eta(t)\int^t_0 [W_R(f)](x,t-t')dt'\right\|_{X_+^{s,\sigma_1}}\lesssim \| f\|_{X_+^{s,\si-1}}+\| f\|_{Z_+^{s,\si-1}}
	\leq  C_1\|u\|^2_{X_+^{s,b}}.
	\end{align*}
	Moreover, according to Lemma \ref{lemma1} , Proposition \ref{bdrh}, and  Theorem \ref{bil},
	\begin{align*}
	\left\|   \eta(t)W_{bdr}(q)  \right\|_{H^{s}_{x}(\m{R}^+)} \lesssim  \|q\|_{H^{\frac{2s+1}{4}}(\R^+)} \lesssim \|f\|_{X_+^{s,\si-1}}+\| f\|_{Z_+^{s,\si-1}}
	\leq  C_1\|u\|^2_{X_+^{s,b}}.	\end{align*}
	Hence,
	$\|\Gamma(u)\|_{H^s(\R^+)}\leq C_2\|(\varphi,\psi,\vec{h})\|_{Y_s}+ 2C_1 \|u\|^2_{X_+^{s,b}}.$
	
	Similarly, based on Lemma \ref{lemma1},  Proposition  \ref{bdrx} and \ref{bdrh}, one has
	\begin{align*}
	\|\Gamma(u)\|_{X_+^{s,b}}
	\lesssim  \|(\varphi,h)\|_{Y_s}+\left\|\eta (t)
	\left(\int^t_0 [W_R(f)](x,t-t')dt'-W_{bdr}(q)\right)\right\|_{X_+^{s,b}}.
	\end{align*}
	Again, according to Lemma \ref{lemma1}, Proposition \ref{kato}, and  Theorem \ref{bil}, one has,
	\begin{align*}
	\left\|\eta(t)
	\int^t_0 [W_R(f)](x,t-\tau)d\tau \right\|_{X_+^{s,b}}\lesssim  \|f\|_{X_+^{s,\si-1}}+\| f\|_{Z_+^{s,\si-1}}
	\leq C_1\|u\|^2_{X_+^{s,b}},
	\end{align*}
	\[
	\left\|\eta (t)[W_{bdr}( q )](x,t)\right\|_{X_+^{s,b}} \lesssim  \|q \|_{H^{\frac{2s+1}{4}}(\R^+)} \leq C_1\|u\|^2_{X_+^{s,b}}. 
	\]
	Thus,
	$\|\Gamma(u)\|_{X_+^{s,b}}\leq C_2\|(\varphi, \vec{h})\|_{Y_s}+2 C_1 \|u\|^2_{X_+^{s,b}}.$
	It follows that
	\[\|\Gamma(u)\|_{\X_+^{s,b}}\leq 2C_2 E_0 + 4C_1 C^2 E^2_0.\]
	By choosing $C=C_2$ and $\eps=\frac{1}{16 C_1C_2^2}$,  it then leads to
$\|\Gamma(u)\|_{\X_+^{s,b}}\leq CE_0.$

	Similar argument can be drawn to establish  $$\|\Gamma(u)-\Gamma(v)\|_{\X_+^{s,b}}\leq \frac12 \|u-v\|_{\X_+^{s,b}}, $$
	with $u,v\in {\X_+^{s,b}}$. Hence, the map $\Gamma$ is a contraction.  Because $\eta\equiv 1$ on $(0,1)$, $u$ is the unique solution to the IBVP   on a time of size $1$. Therefore, we conclude the proof for Theorem.
\end{proof}
\smallskip

\centerline{Acknowledgements}

 S.  Li is supported by the National Natural Science Foundation of China (no. 12001084 and no. 12071061).

\end{document}